\documentclass[11pt]{article}

\usepackage[utf8]{inputenc}
\usepackage[T1]{fontenc}
\usepackage{amsmath,amssymb,amsthm}
\usepackage{geometry}
\usepackage{hyperref}

\newtheorem{theorem}{Theorem}[section]
\newtheorem{lemma}[theorem]{Lemma}
\newtheorem{proposition}[theorem]{Proposition}
\newtheorem{corollary}[theorem]{Corollary}
\newtheorem{definition}[theorem]{Definition}
\newtheorem{fact}[theorem]{Fact}
\theoremstyle{remark}
\newtheorem{remark}[theorem]{Remark}
\newtheorem{example}[theorem]{Example}

\newcommand{\NN}{\mathbb{N}}
\newcommand{\CC}{\mathbb{C}}
\newcommand{\QQ}{\mathbb{Q}}
\newcommand{\ZZ}{\mathbb{Z}}

\title{Differential Operators on $G(r,n)$-Invariant Functions}
\author{Ferdinand Kafando\thanks{Laboratoire de Sciences et Technologies (LaST), Université Thomas SANKARA. \texttt{kafandoferdinand04@gmail.com}} \and Jean Kaboré\thanks{Laboratoire de Sciences et Technologies (LaST), Université Thomas SANKARA. \texttt{kaborejean775075@gmail.com}} \and Ibrahim Nonkané\thanks{Département d'Économie et Mathématiques Appliquées, IUFIC, Université Thomas SANKARA. \texttt{ibrahim.nonkane@uts.bf}}}
\date{\today}

\begin{document}
\maketitle
\tableofcontents

\begin{abstract}
We generalize known results on normalized symmetric coordinates and their dual differential operators, established for the symmetric group, to the complex monomial reflection group $G(r,n):=\mu_r\wr S_n$ (with $G(1,n)=S_n$). Our central tool is a \emph{transfer principle}, which we prove in detail: the substitution $y_i=x_i^r$ identifies the invariant ring of $G(r,n)$ with that of $S_n$, and carries the corresponding operators and coordinates, term by term, into explicit objects in the original variables $x_i$ (rational for the operators, polynomial for the coordinates). From this we deduce the $G(r,n)$-analogues of the known results for $S_n$: a Weyl algebra structure localized at the discriminant of $G(r,n)$, and existence and uniqueness of the dual coordinates $U_k$, the latter following readily from the transfer principle itself. The Leibniz rule for the symmetrized operator is the one exception, since it does not follow directly from the transfer; we give a direct, self-contained proof of it instead. For the main theorem we further give a second, independent proof, which exhibits the underlying linear system as triangular. We then turn to the \emph{total diagonal}, whose preimage splits into $r^{n-1}$ lines permuted transitively by the group. There, unlike the transferred operators $\Delta_i$, the \emph{raw} derivatives $\partial_{x_i}$ exhibit a phenomenon specific to $r\ge2$, which we describe in full via a Faà di Bruno-type structure formula, with an explicit closed formula for its constants. We likewise resolve the degeneracy at an isolated point $x_i=0$ completely: the operator $\Delta_i$ extends holomorphically there, with $\Delta_i\phi=\partial_{i}^r\phi /r!$, from which we deduce a partial analogue of the tangent space to the GIT quotient $\CC^n/G(r,n)$. Only the case of several coordinates vanishing simultaneously remains open, and we delineate it precisely. All proofs are given in detail.
\end{abstract}


\noindent\textbf{2020 AMS Classification:} 13A50, 20F55, 16S32, 05E05, 14L24.

\section{Introduction}\label{sec:intro}

The differentiation of elementary symmetric functions, and of the dual coordinates naturally associated with them, is a classical theme. It sits at the interface of invariant theory, representation theory, and Calogero--Moser systems. Given the normalized elementary symmetric functions in $n$ variables $y_1,\dots,y_n$, Zemel \cite{Ze} constructs a dual family $\{u_k\}_{1\le k\le n}$ of $S_n$-invariant coordinates for the symmetric group $S_n$. These are characterized by the action of symmetrized differential operators $D_d$ satisfying $D_du_k=\delta_{d,k}$. This duality makes $\{u_k\}$ and $\{D_d\}$ the generators of a rank-$n$ Weyl algebra, localized at the usual discriminant $\prod_{i<j}(y_i-y_j)$. It also gives an explicit description of the tangent space to the GIT quotient $\CC^n/S_n$, and a fine analysis of how these coordinates and their derivatives behave at coincidence points, both on the total diagonal and at partial coincidences.

This note generalizes the whole of Section~1 of \cite{Ze}, together with part of its Sections~2 and~3, to the complex monomial reflection group
\[
G(r,n) \;:=\; \mu_r\wr S_n \;=\; \mu_r^n\rtimes S_n,
\]
denoted $G(r,1,n)$ in the Shephard--Todd classification \cite{ST}. This group contains $S_n=G(1,n)$ as a special case, at $r=1$, and for $r\ge2$ covers an infinite family of imprimitive reflection groups occupying a central place in the classification of finite complex reflection groups \cite{ST,Ch}. Our generalization turns on a \emph{transfer principle} (Section~\ref{sec:transfert}): the substitution $y_i=x_i^r$ identifies the invariant ring $\CC[x]^{G(r,n)}$ with $\CC[y]^{S_n}$ (Proposition~\ref{prop:invariants}), and carries any operator identity of \cite{Ze} over, term by term, into a corresponding identity in the original variables $x_i$, typically via the substitutions $\partial_{y_i}\mapsto\Delta_i:=\frac{1}{rx_i^{r-1}}\partial_{x_i}$ and $y_i-y_j\mapsto x_i^r-x_j^r$. This holds for any operator identity expressed in a class of operators $\mathcal L$ closed under sums, products by rational functions, transpositions, and compositions.

This principle does not, however, suffice on its own to transfer every result of \cite{Ze} outright. The Leibniz rule for the symmetrized operator $\mathcal D_d$ needs its own direct argument, independent of \cite{Ze}, which we give in full. The existence-and-uniqueness theorem for the dual coordinates $U_k$, by contrast, follows quickly from the transfer principle together with Zemel's own existence-and-uniqueness result for $S_n$. We give a second, fully self-contained proof of it anyway, since it exhibits the underlying linear system as triangular, which is of independent interest.

Passing from $S_n$ to $G(r,n)$ brings out an entirely new phenomenon, specific to the reflections $x_i\mapsto\zeta x_i$ ($\zeta\in\mu_r$) and absent from the case $r=1$: the ordinary derivative $\partial_{x_i}$, unlike the transferred operator $\Delta_i$, does not behave simply under iteration, and degenerates on the hyperplane $x_i=0$. We give a complete analysis of this. First comes a Faà di Bruno-type structure formula for the iterated raw derivatives, with an explicit closed formula for its constants. Next, their behavior on the total diagonal, whose preimage splits, for $G(r,n)$, into $r^{n-1}$ lines permuted transitively by the group. Finally, the singularity at the origin of an isolated coordinate is fully resolved, with an application to a partial analogue of the tangent space to the GIT quotient $\CC^n/G(r,n)$.

Independently of this, the dual generators $\{U_k\}$ and $\{\mathcal D_d\}$ generate a Weyl algebra localized at the discriminant proper to $G(r,n)$, generalizing the localization result established for $S_n$ in \cite{N}. That some such Weyl-algebra structure must exist in the abstract, for any complex reflection group acting on a polynomial ring, is the content of the Sheppard--Todd--Chevalley theorem, in its differential-operator form: $D(R^G)$ is a Weyl algebra exactly when $G$ is a reflection group (see \cite{Tr}, Theorem~5.3, for an accessible account); our contribution, for $G(r,n)$, is to make this structure fully explicit, with generators given in closed form. An analogous Weyl-algebra-localization phenomenon has been established, with a different set of generators, for the real reflection group of type $B_n$ in \cite{NB}. The group $G(r,n)$ has likewise already been studied from the point of view of $\mathcal D$-modules in \cite{NL}, which decomposes the direct image $\pi_+(\mathcal O_X)$ by means of higher Specht polynomials. Our own approach is complementary to it: we work with the explicit generators $\{U_k\}$ and $\{\mathcal D_d\}$ of the invariant Weyl algebra, rather than with this $\mathcal D$-module decomposition.

The term \emph{symmetric differential operator} is also used, in a quite different sense, in complex analysis and geometric function theory. There it designates operators built from $q$-difference or conformable-fractional derivatives, used to define subclasses of univalent or meromorphically multivalent functions on the unit disc (see, e.g., \cite{IDarus,IAldawish,IEO}). That notion of symmetry refers to a functional-equation symmetry of the operator itself, typically $z\mapsto -z$ or a $q$-analogue thereof, acting on a single complex variable. It is unrelated to the invariant-theoretic symmetry studied here, under a finite reflection group acting on several variables $x_1,\dots,x_n$. We mention this literature only to avoid any terminological ambiguity.

The article is organized as follows. Sections~\ref{sec:rappels} and~\ref{sec:transfert} recall the results of \cite{Ze} that we use, and establish the transfer principle. The following sections develop, for $G(r,n)$, the generators $\mathcal D_I$, the divided-difference relations, the Leibniz rule, the main theorem, and the Weyl algebra structure at the discriminant. They then illustrate all of this on the explicit examples $n=2$ and $n=3$. Section~\ref{sec:diagonale} studies the total diagonal of $G(r,n)$, and the raw derivatives that degenerate there. Section~\ref{sec:trois-points} establishes a closed formula for the structure constants, resolves the degeneracy at the origin of an isolated coordinate, and deduces from it the partial analogue of the tangent space to the GIT quotient. It also delimits precisely what remains open: the simultaneous coincidence of several coordinates at the origin. All proofs are given in detail.

To answer precisely the natural question of what, exactly, is new in passing from $S_n$ to $G(r,n)$, we distinguish four kinds of statements in what follows.

First, the transfer principle itself (Section~\ref{sec:transfert}) is genuinely new. It is a general structural device, with no counterpart in \cite{Ze}, that makes every subsequent transfer possible.

Second, once established, the transfer principle reduces a number of statements to essentially formal consequences of their $S_n$ counterparts: the generators $\mathcal D_I$ and the divided-difference relations of Section~\ref{sec:generators}, the action of $\mathcal D_d$ on the generators $E_h$, the existence and uniqueness of the dual coordinates $U_k$ themselves (Theorem~\ref{thm:main}, via its short first proof), and the vanishing of $U_k$ and of its $\Delta$-derivatives along the total diagonal (the first part of Section~\ref{sec:diagonale}, before the raw derivatives are introduced).

Third, other results demand an argument independent of the transfer principle, even though they run structurally parallel to their $S_n$ counterparts: the $\mu_r^n$-invariance step of Proposition~\ref{prop:invariants}, the Leibniz rule for $\mathcal D_d$ (Proposition~\ref{prop:leibniz}), and the classification of reflections and the discriminant of $G(r,n)$ (Lemma~\ref{lem:disc}). The main theorem admits a second, independent proof of this kind too, given alongside its short transfer-based one for the triangular structure it reveals.

Fourth, several phenomena have no analogue whatsoever for $S_n$, and are trivial, or vacuous, when $r=1$: the $\mu_r^n$-invariance of $\Delta_i$ itself, hence of $\mathcal D_d$ (Proposition~\ref{prop:prop6}), needed for $\mathcal D_d$ to be $G(r,n)$-symmetric, since $S_n$ has no toric factor to check invariance under; the decomposition of the preimage of the total diagonal into $r^{n-1}$ lines (Section~\ref{sec:diagonale}); the whole analysis of the raw derivatives $\partial_{x_i}$, meaning the structure formula, the closed-form constants, and Theorem~\ref{thm:16brut}; and the resolution of the singularity of $\Delta_i$ at an isolated coordinate $x_i=0$ (Section~\ref{sec:trois-points}).

We stress that this last section generalizes only the case of a \emph{single} vanishing coordinate. The case of several coordinates vanishing simultaneously would require the full strength of the coincidence-point analysis of \cite[\S3]{Ze}. It is left open, and stated precisely at the end of Section~\ref{sec:trois-points}.

\section{The group $G(r,n)$ and its ring of invariants}

Before turning to the operators announced in the introduction, we fix the group $G(r,n)$ itself and its ring of polynomial invariants, which everything else in this note builds on. Let $n\ge 1$ and $r\ge 1$ be two integers. Write $\NN_n=\{1,\dots,n\}$ and
\[
G(r,n) \;:=\; \mu_r \wr S_n \;=\; \mu_r^n \rtimes S_n,
\]
the complex reflection group (denoted $G(r,1,n)$ in the Shephard--Todd classification) acting on $\CC^n$ by
\[
(\zeta_1,\dots,\zeta_n;\sigma)\cdot(x_1,\dots,x_n) = (\zeta_1 x_{\sigma^{-1}(1)},\dots,\zeta_n x_{\sigma^{-1}(n)}), \qquad \zeta_i\in\mu_r,\ \sigma\in S_n.
\]
We have $|G(r,n)|=r^n n!$, and $G(1,n)=S_n$, the case treated in full in \cite{Ze}. Throughout, a function $\phi$, defined on a $G(r,n)$-stable subset of $\CC^n$, is called \emph{$G(r,n)$-invariant} if $\phi(g\cdot x)=\phi(x)$ for every $g\in G(r,n)$ and every $x$ in the domain, with the action $g\cdot x$ as just defined. Equivalently, $\phi$ is fixed by the (left) action $(g\cdot\phi)(x):=\phi(g^{-1}\cdot x)$ of $G(r,n)$ on functions, i.e.\ $g\cdot\phi=\phi$ for every $g$. The two formulations coincide because $g$ ranges over the entire group: $g\cdot x=x$ for all $x,g$ is equivalent to $g^{-1}\cdot x=x$ for all $x,g$. All functions considered are holomorphic; we work on the dense open set, which we call the \emph{regular domain},
\[
\Omega := \{x\in\CC^n \mid x_i\neq 0 \text{ for all } i\},
\]
the complement of the union of the reflecting hyperplanes of type $x_i=0$ (see Lemma~\ref{lem:disc}): it is on $\Omega$ that every $\Delta_i$ is defined by its original formula $\Delta_i=\frac{1}{rx_i^{r-1}}\partial_{x_i}$, without ambiguity. Sections~\ref{sec:trois-points} and~\ref{sec:diagonale} will also need to work at points \emph{outside} $\Omega$, where finitely many coordinates vanish; we call this larger set the \emph{extended domain}
\[
\widetilde\Omega \;:=\; \Omega\,\cup\,\{x\in\CC^n : \text{exactly one } x_i=0\} \;=\; \CC^n\setminus\{x: x_i=x_j=0 \text{ for some } i\ne j\},
\]
We reserve $\widetilde\Omega$, throughout, for the set on which the \emph{extended} operator $\Delta_i$ of Theorem~\ref{thm:origine} is defined, holomorphic across an isolated zero $x_i=0$; this is as opposed to the regular domain $\Omega$, on which the original formula for $\Delta_i$ already suffices. We stress that $\widetilde\Omega$ does not cover points where two or more coordinates vanish simultaneously (the origin of $\mathcal T$ in particular), consistently with the scope delimited in Section~\ref{sec:trois-points}.

\begin{proposition}\label{prop:invariants}
The ring of invariants $\CC[x_1,\dots,x_n]^{G(r,n)}$ is the polynomial ring
\[
\CC[x]^{G(r,n)} = \CC[E_1,\dots,E_n], \qquad E_h(x) := e_h(x_1^r,\dots,x_n^r) \quad (1\le h\le n),
\]
where $e_h$ denotes the $h$-th elementary symmetric function. The polynomial $E_h$ is homogeneous of degree $rh$.
\end{proposition}

\begin{proof}
A monomial $x^\alpha$ is invariant under the toric subgroup $\mu_r^n$ if and only if $r\mid \alpha_i$ for all $i$: indeed, for $\zeta=(\zeta_1,\dots,\zeta_n)\in\mu_r^n$, $\zeta\cdot x^\alpha = \big(\prod_{i=1}^n\zeta_i^{\alpha_i}\big)x^\alpha$, and this product equals $1$ for \emph{all} choices of $\zeta\in\mu_r^n$ if and only if it equals $1$ for each $\zeta_i\in\mu_r$ taken separately (the other coordinates of $\zeta$ being fixed to $1$), that is, if and only if $\zeta_i^{\alpha_i}=1$ for all $\zeta_i\in\mu_r$, i.e. $r\mid\alpha_i$, for each $i$. Thus $\CC[x]^{\mu_r^n}=\CC[x_1^r,\dots,x_n^r]=\CC[y_1,\dots,y_n]$ upon setting $y_i:=x_i^r$. Since $G(r,n)=\mu_r^n\rtimes S_n$ with $\mu_r^n$ normal, we get $\CC[x]^{G(r,n)}=\big(\CC[x]^{\mu_r^n}\big)^{S_n}=\CC[y]^{S_n}$, where $S_n$ acts on $y$ by permuting the indices (the action of $S_n\subseteq G(r,n)$ on $x$ by permutation induces exactly the corresponding permutation on $y=x^r$). By the classical structure theorem for the invariant ring of $S_n$ \cite[Ch.~I]{Mac} (recalled in Fact~\ref{fact:A} below, following \cite{Ze}), $\CC[y]^{S_n}=\CC[e_1(y),\dots,e_n(y)]$, a free polynomial ring. Since $e_h(y)=e_h(x^r)=E_h(x)$, the result follows, and $\deg E_h = r\cdot\deg_y e_h = rh$.
\end{proof}

This corresponds exactly to the fundamental degrees $r,2r,\dots,nr$ of $G(r,n)$ given by the Chevalley--Shephard--Todd theorem \cite{ST,Ch}: the invariant ring of a finite complex reflection group is a polynomial ring if and only if the group is generated by reflections. And $G(r,n)$ is indeed generated by reflections: the reflections $t_i^{(\zeta)}$ ($1\le i\le n$, $\zeta$ a generator of $\mu_r$), together with the transpositions $\tau_{i,i+1}^{(1)}$ ($1\le i\le n-1$); both families are classified explicitly in Lemma~\ref{lem:disc} below. The former generate the toric subgroup $\mu_r^n$ coordinatewise, and the latter generate $S_n$ (the standard Coxeter generators); every element $(\zeta;\sigma)=(\zeta;\mathrm{id})\cdot(1;\sigma)$ of $G(r,n)=\mu_r^n\rtimes S_n$ is thus a semidirect product of the two.

Proposition~\ref{prop:invariants} is purely algebraic: it concerns \emph{polynomial} invariants and holds on all of $\CC^n$. The differential-operator constructions of the following sections, however, will require writing a general \emph{holomorphic} $G(r,n)$-invariant function $\phi$ (not assumed polynomial) as $\phi(x)=\eta(x^r)$ for a suitable $S_n$-invariant holomorphic $\eta$. This does not follow from Proposition~\ref{prop:invariants}, and requires a genuine analytic descent statement, valid away from the loci where the group action fails to be free.

\begin{lemma}[Holomorphic descent]\label{lem:descent}
Let $\Omega':=\{x\in\Omega : x_i^r\ne x_j^r \text{ for all } i\ne j\}$ and $Y':=\{y\in\CC^n : y_i\ne0 \text{ and } y_i\ne y_j\ (i\ne j)\}$. Then $G(r,n)$ acts freely on $\Omega'$, the map $x\mapsto x^r$ induces a biholomorphism $\Omega'/G(r,n)\xrightarrow{\ \sim\ }Y'/S_n$, and consequently every $G(r,n)$-invariant holomorphic function $\phi$ on $\Omega'$ can be written, uniquely, as $\phi(x)=\eta(x^r)$ for some $S_n$-invariant holomorphic function $\eta$ on $Y'$.
\end{lemma}

\begin{proof}
\emph{Freeness.} Let $g=(\zeta;\sigma)\in G(r,n)$ fix a point $x\in\Omega'$. If $\sigma$ does not fix every index, some $j\ne i$ has $\sigma(j)=i$, and the fixed-point equation $\zeta_ix_j=x_i$ gives, upon raising to the power $r$ and using $\zeta_i^r=1$, $x_j^r=x_i^r$, which is excluded on $\Omega'$. So $\sigma=\mathrm{id}$, and the equations $\zeta_ix_i=x_i$ with $x_i\ne0$ (since $x\in\Omega$) force $\zeta_i=1$ for every $i$. Hence $g=e$: the action is free. Being finite, $G(r,n)$ then acts properly discontinuously, so the quotient map $\Omega'\to\Omega'/G(r,n)$ is an unramified holomorphic covering, and $\Omega'/G(r,n)$ inherits a unique complex structure for which a function is holomorphic exactly when its pullback to $\Omega'$ is.

\emph{Identification of the quotient.} The map $x\mapsto x^r$ carries $\Omega'$ onto $Y'$, intertwines the $\mu_r^n$-action on $x$ with the trivial action on $y$ (so it is itself an unramified $\mu_r^n$-covering of $Y'$, by the freeness of $\mu_r^n$ on $\Omega'$, a special case of the above), and intertwines the residual $S_n$-action on $x$ with the permutation action on $y$, which is free on $Y'$ (a transposition fixes $y$ only if two of its coordinates coincide, excluded on $Y'$). Composing the two unramified coverings $\Omega'\to Y'\to Y'/S_n$ gives an unramified covering of degree $r^n\cdot n!=|G(r,n)|$ whose fibers are exactly the $G(r,n)$-orbits (since $\mu_r^n$ and $S_n$ together generate $G(r,n)$); this covering therefore coincides with the quotient map for the $G(r,n)$-action, giving the biholomorphism $\Omega'/G(r,n)\cong Y'/S_n$.

\emph{Conclusion.} Under this identification, a $G(r,n)$-invariant holomorphic function on $\Omega'$ is exactly a holomorphic function on $\Omega'/G(r,n)\cong Y'/S_n$, i.e., an $S_n$-invariant holomorphic function $\eta$ on $Y'$, pulled back along $x\mapsto x^r$: this is precisely $\phi(x)=\eta(x^r)$.
\end{proof}

\begin{remark}
Proposition~\ref{prop:invariants} and Lemma~\ref{lem:descent} are complementary, not interchangeable: the former is global and algebraic but restricted to polynomials, the latter is analytic and holds for arbitrary holomorphic functions but only on the dense open subset $\Omega'$. It is Lemma~\ref{lem:descent}, not Proposition~\ref{prop:invariants}, that justifies the descent $\phi=\eta(x^r)$ used for a general holomorphic invariant $\phi$ in the proof of Proposition~\ref{prop:prop5} below and in Corollary~\ref{cor:cor12}.
\end{remark}

\begin{corollary}\label{cor:base}
For $\lambda\vdash k$ a partition of $k$ with $l(\lambda^t)\le n$ (multiplicities $m_h$, $1\le h\le k$), set $E_\lambda:=\prod_{h=1}^k E_h^{m_h}$, homogeneous of degree $rk$ in $x$. The $\{E_\lambda\}_{\lambda\vdash k,\ l(\lambda^t)\le n}$ form a $\ZZ$-basis of the $G(r,n)$-invariant homogeneous polynomials of degree $rk$ with integer coefficients.
\end{corollary}

\begin{proof}
The ring morphism $\sigma:\ZZ[y_1,\dots,y_n]\to\ZZ[x_1,\dots,x_n]$, $y_i\mapsto x_i^r$, is injective (the $x_i^r$ are algebraically independent, just as the $x_i$ themselves are), and sends a monomial $y^\beta$ to $x^{r\beta}$ \emph{with the same coefficient} (no rescaling): it is thus an isomorphism of $\ZZ$-modules between $\ZZ[y]_k$ and the sub-$\ZZ$-module of $\ZZ[x]_{rk}$ generated by the monomials $\{x^{r\beta} : |\beta|=k\}$. By Fact~\ref{fact:A} (via \cite[Ch.~I]{Mac}), $\{e_\lambda(y)\}_{\lambda\vdash k}$ is a $\ZZ$-basis of $\ZZ[y]^{S_n}_k$; its image $\{E_\lambda(x)=\sigma(e_\lambda(y))\}_{\lambda\vdash k}$ is thus a $\ZZ$-basis of the $\ZZ$-module $\sigma\big(\ZZ[y]^{S_n}_k\big)$.

It remains to identify this last module with $\ZZ[x]^{G(r,n)}_{rk}$. The inclusion $\sigma(\ZZ[y]^{S_n}_k)\subseteq\ZZ[x]^{G(r,n)}_{rk}$ is immediate (the image of an $S_n$-symmetric invariant is $G(r,n)$-invariant, as in the proof of Proposition~\ref{prop:invariants}). Conversely, let $\phi\in\ZZ[x]^{G(r,n)}_{rk}$; by Proposition~\ref{prop:invariants} (valid over $\CC$), $\phi=\eta(x^r)$ for a unique $\eta\in\CC[y]^{S_n}_k$. Since $\phi$ is $G(r,n)$-invariant, it only involves monomials $x^{r\beta}$ (no other monomial could appear in a polynomial written solely in the $y_i=x_i^r$); by injectivity of $\sigma$ and the fact that it preserves coefficients exactly, monomial by monomial, $\eta$ therefore has the \emph{same} integer coefficients as $\phi$ (via $\beta\leftrightarrow r\beta$), so that $\eta\in\ZZ[y]_k$; and $\eta$ is $S_n$-invariant by the same argument as in the proof of Proposition~\ref{prop:invariants}. Thus $\phi\in\sigma(\ZZ[y]^{S_n}_k)$, which completes the proof.
\end{proof}

\section{Recollections for $S_n$ and the transfer principle}

This section assembles the tools needed to transport results about $G(r,n)$'s invariant ring from the better-understood case of $S_n$: Zemel's operators and dual coordinates for $S_n$, recalled below, and the substitution $y_i=x_i^r$ that connects the two settings, developed into the transfer principle in Section~\ref{sec:transfert}.

\subsection{Generators $D_I$ and dual coordinates $u_k$ for $S_n$}\label{sec:rappels}

To keep this note self-contained, we recall here the precise statements from Section~1 of \cite{Ze} that we use, without proof; the proofs themselves are found in \cite{Ze}. Here $y_1,\dots,y_n$ denote the variables, $e_h:=e_h(y_1,\dots,y_n)$, $\tilde e_h:=\frac{(n-h)!}{n!}e_h$, and for $I\subseteq\NN_n$, $D_I:=\sum_{i\in I}\partial_{y_i}/\prod_{i\ne j\in I}(y_j-y_i)$.

\begin{fact}[Theorem~1 of \cite{Ze}]\label{fact:A}
The products $\{e_\lambda\}_{\lambda\vdash k,\ l(\lambda^t)\le n}$ form a $\ZZ$-basis of the homogeneous symmetric polynomials of degree $k$ in the $y_i$ with integer coefficients. The ring $\ZZ[y_1,\dots,y_n]^{S_n}$ is generated, as a $\ZZ$-algebra, by $\{e_h\}_{h=1}^n$. This result is classical: \cite{Ze} itself refers to Macdonald's reference treatise \cite[Ch.~I]{Mac} for its proof, so we cite that directly, rather than \cite{Ze}, whenever this underlying fact is invoked.
\end{fact}

\begin{fact}[Lemma~2]\label{fact:B}
For $I\subseteq\NN_n$, $e_h(y_{\NN_n}) = \sum_{l=0}^h e_l(y_I)\,e_{h-l}(y_{I^c})$.
\end{fact}

\begin{fact}[Lemma~3]\label{fact:C}
For $I$ nonempty, $|I|=d$: $D_I\,e_l(y_I) = \delta_{l,d}$.
\end{fact}

\begin{fact}[Corollary~4]\label{fact:D}
$D_I\,e_h = e_{h-d}(y_{I^c})$ for $h\ge d$ (and $0$ for $h<d$, with $e_l:=0$ if $l<0$).
\end{fact}

\begin{fact}[Proposition~5]\label{fact:E}
For $I\subseteq\NN_n$, $i\notin I$, $j\in I$, $J=I\cup\{i\}$, and $\eta$ symmetric: $\delta_{ij}D_I\eta = D_J\eta$, where $\delta_{ij}\eta:=(\eta-s_{ij}\eta)/(y_i-y_j)$.
\end{fact}

\begin{fact}[Proposition~6]\label{fact:F}
$D_d:=d!\sum_{|I|=d}D_I$ sends $e_h$ ($h\ge d$) to $\frac{(n-h+d)!}{(n-h)!}e_{h-d}$, and $0$ if $h<d$.
\end{fact}

\begin{fact}[Corollary~7]\label{fact:G}
$D_d\,\tilde e_h = \tilde e_{h-d}$ ($0\le h\le n$, zero if $h<d$).
\end{fact}

\begin{fact}[Corollary~8]\label{fact:H}
For $\lambda\vdash k$ (multiplicities $m_h$): $D_d\,\tilde e_\lambda = \sum_{h=d}^k m_h\,\tilde e_{\lambda-d\varepsilon_h}$, where $\tilde e_\lambda:=\prod_h \tilde e_h^{m_h}$ and $\lambda-d\varepsilon_h$ is the partition of $k-d$ obtained by replacing a part $h$ of $\lambda$ with $h-d$.
\end{fact}

\begin{fact}[Theorem~10]\label{fact:I}
For every $1\le k\le n$, there exists a unique $u_k\in\QQ[y]^{S_n}$, homogeneous of degree $k$, such that $D_d u_k=\delta_{d,k}$ for all $d$; explicitly,
$u_k = \sum_{\lambda\vdash k}(-1)^{l(\lambda)-1}\frac{(l(\lambda)-1)!}{\prod_h m_h!}\,\tilde e_\lambda$.
\end{fact}

\begin{fact}[Corollary~12]\label{fact:J}
If $\eta(y_1,\dots,y_n)=\psi(u_1,\dots,u_n)$ is symmetric, then $\psi_{u_d}=D_d\eta$ for every $d$ (wherever the $y_i$ are pairwise distinct).
\end{fact}

\begin{fact}[Theorem~16]\label{fact:K}
Every single derivative $\prod_{q=1}^d\partial_{y_{i_q}}$ of order $d$ (the indices $i_q$ possibly coinciding), applied to $u_k$ with $d\ne k$, vanishes at every point of the total diagonal $\{y_1=\cdots=y_n\}$.
\end{fact}

\begin{fact}[Corollary~18 and Theorem~21 (normalization value)]\label{fact:L}
At the point of the total diagonal where $y_i=b$ for all $i$: $u_1=b$ and $u_k=0$ for $k\ge2$. The total diagonal is characterized algebraically, in $\mathrm{Sym}^n\CC$, by the vanishing of $\{u_k\}_{k=2}^n$. Moreover $\partial_{y_i}^k u_k$ depends neither on $b$ nor on $i$: it is a constant \emph{on all of $\CC^n$}, since a homogeneous polynomial of degree $0$ is necessarily constant, $u_k$ being homogeneous of degree $k$; and, in the normalization of \cite{Ze}, equals
\[
\partial_{y_i}^k u_k \;=\; \frac{(-1)^{k-1}(k-1)!}{n^k}.
\]
\end{fact}

\subsection{The transfer lemma and principle}\label{sec:transfert}

We can now build the substitution that carries these facts about $S_n$ over to $G(r,n)$. On $\Omega$, define the order-$1$ rational differential operator
\[
\Delta_i := \frac{1}{r\,x_i^{r-1}}\,\partial_{x_i}, \qquad 1\le i\le n.
\]

\begin{lemma}[Transfer lemma]\label{lem:transfert}
Let $x^0\in\Omega$ and let $\eta$ be a holomorphic function near $y^0:=(x^0)^r$. Set $\phi(x):=\eta(x^r)$ near $x^0$ (with $x^r:=(x_1^r,\dots,x_n^r)$). Then, for every $i$,
\[
\Delta_i \phi = (\partial_{y_i}\eta)(x^r).
\]
\end{lemma}

\begin{proof}
By the chain rule, $\partial_{x_i}\phi(x) = r x_i^{r-1}\,(\partial_{y_i}\eta)(x^r)$ (only the $i$-th variable of $\eta$ is affected, since $x^r=(x_1^r,\dots,x_n^r)$ depends on $x_i$ only through $y_i$). Dividing by $rx_i^{r-1}$ gives the statement.
\end{proof}

\begin{example}[Direct verification for $\Delta_i^2$]\label{ex:delta2}
Lemma~\ref{lem:transfert} gives, in particular, $\Delta_i(\Delta_i\phi)=(\partial_{y_i}^2\eta)(x^r)$ (applying it a second time to $\phi^{(1)}:=\Delta_i\phi=\eta^{(1)}(x^r)$, $\eta^{(1)}:=\partial_{y_i}\eta$). It is instructive to check this by a direct computation: writing $c(x_i):=\frac{1}{r}x_i^{1-r}$, we have $\Delta_i\phi=c\,\partial_i\phi$, hence
\[
\Delta_i^2\phi = c\,\partial_i(c\,\partial_i\phi) = c^2\,\partial_i^2\phi + c\,c'\,\partial_i\phi.
\]
Setting $\eta^{(1)}=\partial_{y_i}\eta$, $\eta^{(2)}=\partial_{y_i}^2\eta$ (evaluated at $x^r$), we have $\partial_i\phi = rx_i^{r-1}\eta^{(1)}$ and $\partial_i^2\phi = r(r-1)x_i^{r-2}\eta^{(1)}+r^2x_i^{2r-2}\eta^{(2)}$ (chain rule applied twice). Expanding then gives
\[
c^2\partial_i^2\phi = \frac{r-1}{r\,x_i^r}\eta^{(1)} + \eta^{(2)}, \qquad c\,c'\,\partial_i\phi = -\frac{r-1}{r\,x_i^r}\eta^{(1)},
\]
whose sum is exactly $\eta^{(2)}(x^r)$, confirming that the correction terms cancel exactly, as announced by Lemma~\ref{lem:transfert}.
\end{example}

Denote by $\mathcal L$ the class of operators, acting on holomorphic functions of $y=(y_1,\dots,y_n)$, built inductively from the derivations $\partial_{y_i}$, multiplications by rational functions, and transpositions of variables, by:
\begin{itemize}
\item[(i)] (base operators) $\partial_{y_i}$ for $1\le i\le n$, and multiplication by a rational function $q(y)\in\CC(y_1,\dots,y_n)$;
\item[(ii)] (transposition) $s_{ij}$, $1\le i<j\le n$, exchanging the variables $y_i,y_j$, acting by $\eta\mapsto s_{ij}\eta:=\eta\circ s_{ij}$;
\item[(iii)] (sum) if $L_1,L_2\in\mathcal L$, then $L_1+L_2\in\mathcal L$;
\item[(iv)] (composition) if $L_1,L_2\in\mathcal L$, then $L_1\circ L_2\in\mathcal L$ (acting by $\eta\mapsto L_1(L_2\eta)$; this includes the ordinary product when $L_2$ is a multiplication).
\end{itemize}
To each $L\in\mathcal L$ we associate its \emph{transfer} $\widetilde L$, obtained by replacing everywhere $\partial_{y_i}\to\Delta_i$, $q(y)\to q(x^r)$ (multiplication by the transferred rational function), and $s_{ij}\to\sigma_{ij}$ (transposition of $x_i,x_j$), with sums and compositions preserved.

\begin{proposition}[Transfer principle]\label{prop:transfert}
For every $L\in\mathcal L$ and every $\eta$ holomorphic on an open set containing $\sigma(y^0)$ for every $\sigma$ in the subgroup of $S_n$ generated by the transpositions $s_{ij}$ occurring in the construction of $L$ (in particular $y^0=(x^0)^r$ itself, taking $\sigma=\mathrm{id}$), and large enough for $L\eta$ to be defined (an automatic condition if $\eta$ is holomorphic on all of $\CC^n$, as is the case in every application in this note, where $\eta$ is a polynomial or a global rational function), we have, for $\phi:=\eta(x^r)$,
\[
\widetilde L\,\phi = (L\eta)(x^r).
\]
\end{proposition}

\begin{proof}
Structural induction on the construction of $L$.

\emph{Base case.} If $L=\partial_{y_i}$, this is Lemma~\ref{lem:transfert}. If $L=q(y)$ (multiplication), then $\widetilde L\phi = q(x^r)\,\eta(x^r) = (q\eta)(x^r) = (L\eta)(x^r)$, trivially. If $L=s_{ij}$ (transposition), then $\widetilde L=\sigma_{ij}$ and, for $\hat x:=\sigma_{ij}\cdot x$ (the coordinates $x_i,x_j$ exchanged), we have $\hat x^r = s_{ij}(x^r)$ (exchanging $x_i,x_j$ and then raising to the power $r$ gives the same result as raising to the power $r$ and then exchanging $y_i=x_i^r,y_j=x_j^r$); $\phi(\hat x)=\eta(\hat x^r)$ is defined near $x^0$ because $\eta$ is, by hypothesis, holomorphic near $s_{ij}(y^0)$. Hence $\widetilde L\phi = \sigma_{ij}\phi = \phi(\hat x) = \eta(\hat x^r) = \eta(s_{ij}(x^r)) = (s_{ij}\eta)(x^r) = (L\eta)(x^r)$.

\emph{Sum.} If $L=L_1+L_2$ and the statement holds for $L_1,L_2$, then by linearity of $\widetilde L=\widetilde{L_1}+\widetilde{L_2}$ and of $L=L_1+L_2$:
\[
\widetilde L\phi = \widetilde{L_1}\phi+\widetilde{L_2}\phi = (L_1\eta)(x^r)+(L_2\eta)(x^r) = ((L_1+L_2)\eta)(x^r) = (L\eta)(x^r).
\]

\emph{Composition.} If $L=L_1\circ L_2$ and the statement holds for $L_1$ and for $L_2$ (applied to \emph{any} admissible function, in particular to $L_2\eta$), set $\eta_2:=L_2\eta$ and $\phi_2:=\eta_2(x^r)$. By the induction hypothesis applied to $L_2$: $\widetilde{L_2}\phi=(L_2\eta)(x^r)=\phi_2$. By the induction hypothesis applied to $L_1$ (with $\eta_2$ in place of $\eta$): $\widetilde{L_1}\phi_2 = (L_1\eta_2)(x^r) = (L_1(L_2\eta))(x^r)=(L\eta)(x^r)$. But $\widetilde{L_1}\phi_2 = \widetilde{L_1}(\widetilde{L_2}\phi) = (\widetilde{L_1}\circ\widetilde{L_2})\phi = \widetilde L\,\phi$ by definition of the transfer of a composition. Hence $\widetilde L\phi=(L\eta)(x^r)$.

These four cases exhaust the inductive construction of $\mathcal L$, which completes the induction.
\end{proof}

\begin{remark}[On the holomorphy hypothesis]\label{rem:holomorphy}
The hypothesis must hold at $\sigma(y^0)$ for \emph{every} $\sigma$ in the subgroup generated by the transpositions occurring in $L$, not merely at $s_{ij}(y^0)$ for each such transposition individually: for $L=s_{12}\circ s_{23}$, say, the Composition step of the proof applies the induction hypothesis to $\eta_2:=s_{23}\eta$ in place of $\eta$ for $L_1=s_{12}$, which requires $\eta_2$ holomorphic near $s_{12}(y^0)$, i.e.\ $\eta$ holomorphic near $s_{23}(s_{12}(y^0))$, a composite point not covered by evaluating each transposition individually at $y^0$. Since the generated subgroup is closed under products, any such composite point, at any depth of the induction, already lies in the required domain. This refinement affects only the fully general statement of the proposition: every application of the transfer principle in this note ($\mathcal D_I$, $\delta_{ij}\mathcal D_I$, $\mathcal D_d$) involves at most a single transposition, for which the two formulations coincide, so none of the results below are affected.
\end{remark}

\begin{remark}\label{rem:transfertusage}
The operators $D_I$, $D_d$, and $\delta_{ij}$ of Section~1 of \cite{Ze} belong to $\mathcal L$: they are sums of compositions of finitely many $\partial_{y_i}$, transpositions $s_{ij}$ (for $\delta_{ij}$), and multiplications by rational functions of $y$. Proposition~\ref{prop:transfert} therefore applies to each of them, for every admissible function $\eta$, which rigorously justifies every transfer performed in what follows.

In particular, applied to $L=D_d\in\mathcal L$ and $\eta=u_k$ (Fact~\ref{fact:I}), it gives $\mathcal D_d\big(u_k(x^r)\big) = (D_du_k)(x^r) = \delta_{d,k}(x^r)=\delta_{d,k}$: the function $x\mapsto u_k(x^r)$, which is $G(r,n)$-invariant and homogeneous of degree $rk$ in $x$ (since $u_k$ is $S_n$-invariant and homogeneous of degree $k$ in $y$), thus satisfies exactly the characteristic property of $U_k$ stated in Theorem~\ref{thm:main}. By the uniqueness established in that theorem, it follows that
\[
U_k(x) = u_k(x^r) \qquad (1\le k\le n),
\]
an identity that will be used in Section~\ref{sec:diagonale}.
\end{remark}

\begin{remark}[Summary of notation]\label{rem:notation}
For reference, the differential operators used throughout this note are, in the $S_n$/$y$-side notation of \cite{Ze} (Section~\ref{sec:rappels}) on the left, and their $G(r,n)$/$x$-side transfers (from Section~\ref{sec:generators} on) on the right:
\begin{center}
\small
\begin{tabular}{@{}p{0.44\linewidth}p{0.5\linewidth}@{}}
$y$-side (Section~\ref{sec:rappels}) & $x$-side transfer \\ \hline
$\partial_{y_i}$, first-order derivative & $\Delta_i:=\dfrac{1}{rx_i^{r-1}}\partial_{x_i}$ (Section~\ref{sec:intro}) \\[2pt]
$D_I=\sum_{i\in I}\partial_{y_i}/\prod_{i\ne j\in I}(y_j-y_i)$ & $\mathcal D_I$ (Lemma~\ref{lem:lem3}) \\[2pt]
$D_d=d!\sum_{|I|=d}D_I$, symmetrized order $d$ & $\mathcal D_d=d!\sum_{|I|=d}\mathcal D_I$ (Prop.~\ref{prop:leibniz} on) \\[2pt]
$\delta_{ij}\eta=(\eta-s_{ij}\eta)/(y_i-y_j)$, divided difference & transferred via $\mathcal D_I\leftrightarrow\mathcal D_J$ (Section~\ref{sec:generators}) \\[2pt]
$u_k$, $e_h$, $\widetilde e_h$: dual/elementary/normalized coordinates & $U_k$, $E_h$, $\widetilde E_h$, via $U_k(x)=u_k(x^r)$
\end{tabular}
\end{center}
The operators $\partial_{H_\alpha}^d$ of Fact~\ref{fact:M} (Section~\ref{sec:trois-points}), restricted to a block of variables, and their transfer $\Delta_{H_\alpha}^d$, follow the same pattern with $n$ replaced by $|H_\alpha|$; the null-block operator $\Delta_{i_0}$ of Theorem~\ref{thm:origine} has no $y$-side or $S_n$ counterpart, being specific to $G(r,n)$.
\end{remark}

\section{Generators \texorpdfstring{$\mathcal D_I$}{DI} and divided-difference relations}\label{sec:generators}

The first family of operators to put the transfer principle to work on is the $D_I$ from Section~\ref{sec:rappels}; their transfer to $G(r,n)$ underlies everything that follows.

\subsection{Generators and the operators \texorpdfstring{$\mathcal D_I$}{DI}}

For $I\subseteq\NN_n$, we generalize the notation $E_h=E_h(x_{\NN_n})$ of Section~\ref{sec:rappels} by setting $E_h(x_I):=e_h\big((x_i^r)_{i\in I}\big)$, in agreement with the notation $e_h(y_I)$ of Fact~\ref{fact:B}.

\begin{lemma}\label{lem:lem2}
For $I\subseteq\NN_n$ and $1\le h\le n$,
\[
E_h(x_{\NN_n}) = \sum_{l=0}^h E_l(x_I)\,E_{h-l}(x_{I^c}).
\]
\end{lemma}

\begin{proof}
Fact~\ref{fact:B} is a polynomial identity valid for every $y\in\CC^n$, hence in particular at the point $y=x^r$; since no differential operator is involved, it suffices to substitute $y=x^r$ directly (without invoking Proposition~\ref{prop:transfert}, which is reserved for the transfer of operators). For any subset $J\subseteq\NN_n$, $e_h(y_J)|_{y=x^r}=e_h\big((x_i^r)_{i\in J}\big)=E_h(x_J)$ by the definition above; applied successively to $J=\NN_n$, $I$, and $I^c$, this turns the identity of Fact~\ref{fact:B} into the announced identity.
\end{proof}

\begin{lemma}\label{lem:lem3}
Let $I\subseteq\NN_n$ be nonempty, $d:=|I|$. The operator
\[
\mathcal D_I \;:=\; \sum_{i\in I} \frac{\Delta_i}{\displaystyle\prod_{i\neq j\in I}(x_j^r-x_i^r)}
\]
sends $E_l(x_I)$ to $\delta_{l,d}$.
\end{lemma}

\begin{proof}
$D_I=\sum_{i\in I}\partial_{y_i}/\prod_{i\ne j\in I}(y_j-y_i)\in\mathcal L$ (a sum of compositions of $\partial_{y_i}$ with multiplications by rational functions), and its transfer is exactly $\mathcal D_I$ (since $y_j-y_i\to x_j^r-x_i^r$). By Proposition~\ref{prop:transfert} applied to $\eta=e_l(y_I)$, and Fact~\ref{fact:C} ($D_Ie_l(y_I)=\delta_{l,d}$, a constant), we obtain $\mathcal D_I E_l(x_I) = \delta_{l,d}(x^r)=\delta_{l,d}$.
\end{proof}

\begin{corollary}\label{cor:cor4}
$\mathcal D_I$ sends $E_h=E_h(x_{\NN_n})$ to $E_{h-d}(x_{I^c})$ if $h\ge d$, and to $0$ if $h<d$ (with $E_l:=0$ for $l<0$).
\end{corollary}

\begin{proof}
Combine Lemma~\ref{lem:lem2} and Lemma~\ref{lem:lem3}. Each $\Delta_i$ ($i\in I$) satisfies the ordinary Leibniz rule (it is $\partial_{x_i}$ up to a factor, cf.\ Example~\ref{ex:delta2}) and annihilates any factor independent of $x_i$; since $E_{h-l}(x_{I^c})$ does not depend on the $x_i,\,i\in I$, applying $\mathcal D_I$ to the decomposition of Lemma~\ref{lem:lem2} gives
\[
\mathcal D_I E_h = \sum_{l=0}^h \big(\mathcal D_I E_l(x_I)\big)\,E_{h-l}(x_{I^c}) = \sum_{l=0}^h \delta_{l,d}\,E_{h-l}(x_{I^c})
\]
by Lemma~\ref{lem:lem3}. Only the term $l=d$ contributes: if it occurs in the sum (i.e.\ $h\ge d$), we obtain $E_{h-d}(x_{I^c})$; otherwise ($h<d$), every term is zero.
\end{proof}

\subsection{Divided-difference relations}

A divided-difference relation links the transferred generators $\mathcal D_I$, letting us pass from one to another; that relation comes next. For $i\neq j$, let $\sigma_{ij}$ denote the transposition of $x_i,x_j$ (an element of $G(r,n)$, with trivial torsion); it induces on $y=x^r$ the exchange $y_i\leftrightarrow y_j$. For any function $\phi$ of the $x$'s, not necessarily invariant, set
\[
\delta_{ij}\phi := \frac{\phi - \sigma_{ij}\phi}{x_i^r - x_j^r}.
\]

\begin{proposition}\label{prop:prop5}
Let $I\subseteq\NN_n$, $i\notin I$, $j\in I$, $J:=I\cup\{i\}$. For every $G(r,n)$-invariant holomorphic function $\phi$ on $\Omega'$ (Lemma~\ref{lem:descent}),
\[
\delta_{ij}\,\mathcal D_I\phi = \mathcal D_J\phi .
\]
\end{proposition}

\begin{proof}
Write $\phi=\eta(x^r)$ with $\eta$ an $S_n$-invariant holomorphic function on $Y'$ (Lemma~\ref{lem:descent}, \emph{not} Proposition~\ref{prop:invariants}, which only concerns polynomials). The operator $\eta\mapsto \delta_{ij}D_I\eta$ belongs to $\mathcal L$. Indeed, $\delta_{ij}=(\mathrm{id}-s_{ij})/(y_i-y_j)$ is a sum (case iii) of multiplication by $1/(y_i-y_j)$ (case i) and its composition with $-s_{ij}$ (cases ii, iv), so $\delta_{ij}D_I\in\mathcal L$ by composition with $D_I\in\mathcal L$ (case iv). Its transfer is $\delta_{ij}\mathcal D_I$, with $\delta_{ij}$ transferred term by term to $(\mathrm{id}-\sigma_{ij})/(x_i^r-x_j^r)$, in agreement with the definition of $\delta_{ij}$ on $x$ given above. Fact~\ref{fact:E} gives $\delta_{ij}D_I\eta=D_J\eta$; transferring via Proposition~\ref{prop:transfert}, $\delta_{ij}\mathcal D_I\phi = (\delta_{ij}D_I\eta)(x^r) = (D_J\eta)(x^r) = \mathcal D_J\phi$.
\end{proof}

\section{The symmetrized operator and the main theorem}

The generators $\mathcal D_I$ are indexed by subsets of a fixed size $d$; summing them over all such subsets produces the single operator $\mathcal D_d$ at the center of this section, together with the dual coordinates $U_k$ it characterizes.

\subsection{The symmetrized operator, its normalization, and the Leibniz rule}

We say that an operator $T$ acting on smooth functions on $\Omega$ is \emph{$G(r,n)$-symmetric} if it commutes with the action of $G(r,n)$, i.e.\ $T(g\cdot\phi)=g\cdot(T\phi)$ for every $g\in G(r,n)$ and every admissible $\phi$, where $(g\cdot\phi)(x):=\phi(g^{-1}x)$ is the same left action on functions already introduced in Section~\ref{sec:intro} (for a transposition $\sigma_{ij}$, an involution equal to its own inverse, this agrees with the map $\phi\mapsto\phi\circ\sigma_{ij}$ used for it in Section~\ref{sec:transfert}). Since $G(r,n)=\mu_r^n\rtimes S_n$ is generated by $S_n$ and by the torus $\mu_r^n$, it suffices to check that $T$ commutes with each of these two subgroups.

\begin{proposition}\label{prop:prop6}
For $1\le d\le n$, the operator
\[
\mathcal D_d := d!\sum_{\substack{I\subseteq\NN_n\\ |I|=d}} \mathcal D_I
\]
is a $G(r,n)$-symmetric differential operator, which sends $E_h$ ($h\ge d$) to $\dfrac{(n-h+d)!}{(n-h)!}E_{h-d}$, and to $0$ if $h<d$.
\end{proposition}

\begin{proof}
Symmetry is immediate: a permutation $\sigma\in S_n\subseteq G(r,n)$ permutes the size-$d$ subsets $I$, so the sum $\sum_{|I|=d}\mathcal D_I$ is invariant under conjugation by $S_n$. Invariance under the toric part $\mu_r^n$ follows from a direct computation: under $x_i\mapsto\zeta x_i$ ($\zeta\in\mu_r$), $\partial_{x_i}\mapsto\zeta^{-1}\partial_{x_i}$ and $x_i^{r-1}\mapsto\zeta^{r-1}x_i^{r-1}$, hence $\Delta_i=\frac{1}{rx_i^{r-1}}\partial_{x_i}\mapsto \zeta^{-r}\Delta_i=\Delta_i$ since $\zeta^r=1$. Thus each $\Delta_i$, and hence $\mathcal D_I$, commutes with the toric action. The action on $E_h$ follows from the transfer of Fact~\ref{fact:F}, since the operator $D_d=d!\sum_{|I|=d}D_I\in\mathcal L$ transfers to $\mathcal D_d$; or directly from Corollary~\ref{cor:cor4}, by summing over the $\binom{n-h+d}{d}\cdot d!$ relevant subsets $I$, exactly as in the proof of the corresponding result of \cite{Ze}.
\end{proof}

We normalize, as in \cite{Ze},
\[
\widetilde E_h := \frac{(n-h)!}{n!}\,E_h \qquad (0\le h\le n,\ \widetilde E_0:=1).
\]

\begin{corollary}\label{cor:cor7}
$\mathcal D_d\,\widetilde E_h = \widetilde E_{h-d}$ for $0\le h\le n$ (zero if $h<d$).
\end{corollary}

\begin{proof}
By Proposition~\ref{prop:prop6}, $\mathcal D_d\widetilde E_h = \frac{(n-h)!}{n!}\cdot\frac{(n-h+d)!}{(n-h)!}E_{h-d} = \frac{(n-h+d)!}{n!}E_{h-d} = \widetilde E_{h-d}$ (and $0$ if $h<d$), a computation identical to that of Corollary~7 of \cite{Ze}.
\end{proof}

For a partition $\lambda\vdash k$ (multiplicities $m_h$, $1\le h\le k$), set $\widetilde E_\lambda := \prod_{h=1}^{k}\widetilde E_h^{\,m_h}$.

\begin{proposition}[Leibniz rule for $\mathcal D_d$]\label{prop:leibniz}
$\mathcal D_d$ ($1\le d\le n$) is a \emph{rational} differential operator on $\Omega$, with poles on the locus $\{x_i^r=x_j^r,\ \exists\, i\ne j\}$ (arising from the denominators of its terms $\mathcal D_I$, Lemma~\ref{lem:lem3}; this locus is empty for $d=1$, where $\mathcal D_1=\sum_i\Delta_i$ has no denominator). For all smooth $\phi,\eta$ on $\Omega$ (not necessarily invariant), at every point where $\mathcal D_d$ is defined,
\[
\mathcal D_d(\phi\eta) = (\mathcal D_d\phi)\,\eta + \phi\,(\mathcal D_d\eta).
\]
\end{proposition}

\begin{proof}
By construction, $\mathcal D_I=\sum_{i\in I} f_{I,i}(x)\,\partial_{x_i}$ with $f_{I,i}(x):=\big[rx_i^{r-1}\prod_{i\ne j\in I}(x_j^r-x_i^r)\big]^{-1}$ a rational function. Hence $\mathcal D_d=d!\sum_{|I|=d}\mathcal D_I=\sum_{i=1}^n g_i(x)\,\partial_{x_i}$, where $g_i(x):=d!\sum_{|I|=d,\,i\in I} f_{I,i}(x)$ is again a rational function. An operator of the form $\sum_i g_i(x)\partial_{x_i}$ (a vector field, with no order-$0$ term) trivially satisfies the general Leibniz rule, by linearity and the ordinary Leibniz rule for each $\partial_{x_i}$:
\[
\mathcal D_d(\phi\eta) = \sum_i g_i\,\partial_i(\phi\eta) = \sum_i g_i\big(\eta\,\partial_i\phi + \phi\,\partial_i\eta\big) = \eta\sum_i g_i\partial_i\phi + \phi\sum_i g_i\partial_i\eta = \eta\,\mathcal D_d\phi + \phi\,\mathcal D_d\eta. \qedhere
\]
\end{proof}

\begin{corollary}\label{cor:cor8}
For $1\le d\le n$ and $\lambda\vdash k$,
\[
\mathcal D_d\, \widetilde E_\lambda = \sum_{h=d}^{k} m_h\, \widetilde E_{\lambda-d\varepsilon_h}
\]
(an empty, hence zero, sum if $d>k$).
\end{corollary}

\begin{proof}
By Proposition~\ref{prop:leibniz} (the generalized Leibniz rule, applied iteratively to the product $\widetilde E_\lambda=\prod_{h=1}^k\widetilde E_h^{m_h}$),
\[
\mathcal D_d\widetilde E_\lambda = \sum_{h=1}^{k} m_h\,\widetilde E_h^{\,m_h-1}\Big(\prod_{g\ne h}\widetilde E_g^{\,m_g}\Big)\cdot \mathcal D_d\widetilde E_h.
\]
By Corollary~\ref{cor:cor7}, $\mathcal D_d\widetilde E_h=0$ for $h<d$, and $=\widetilde E_{h-d}$ for $h\ge d$. The terms $h<d$ therefore vanish, and for $h\ge d$ the corresponding term equals $m_h\,\widetilde E_h^{m_h-1}\widetilde E_{h-d}\prod_{g\ne h}\widetilde E_g^{m_g} = m_h\,\widetilde E_{\lambda-d\varepsilon_h}$, since replacing, in the product $\widetilde E_\lambda$, one factor $\widetilde E_h$ (among the $m_h$ present) by $\widetilde E_{h-d}$ is exactly the definition of $\widetilde E_{\lambda - d\varepsilon_h}$. Summing over $h=d,\dots,k$ gives the announced equality, a priori only outside the polar locus of $\mathcal D_d$ (Proposition~\ref{prop:leibniz}). But $\mathcal D_d\widetilde E_\lambda$ is, by construction (Lemma~\ref{lem:lem3}), a rational function on $\Omega$ whose possible poles are contained in the finite union of the hyperplanes $\{x_i^r=x_j^r\}$; since it coincides, on the dense complement of that union, with the polynomial $\sum_hm_h\widetilde E_{\lambda-d\varepsilon_h}$, the equality between rational functions extends to all of $\Omega$. This shows, in passing, that $\mathcal D_d\widetilde E_\lambda$ actually has no poles at all.
\end{proof}

\begin{remark}
This proof is entirely direct: it does not invoke the transfer principle, only the immediate fact that $\mathcal D_d$ is a derivation (a rational vector field) on $\Omega$, together with the elementary extension argument for rational functions (two rational functions agreeing on a dense open set agree wherever both are defined) to extend the identity, a priori valid only outside the polar locus of $\mathcal D_d$, to all of $\Omega$.
\end{remark}

\subsection{The main theorem}

Existence and uniqueness of the dual coordinates $U_k$ themselves can now be stated and proved, with $\mathcal D_d$ and its Leibniz rule in place.

\begin{theorem}\label{thm:main}
For every $1\le k\le n$, there exists a unique $G(r,n)$-invariant polynomial $U_k$, homogeneous of degree $rk$, such that $\mathcal D_d\, U_k = \delta_{d,k}$ for every $1\le d\le n$. It is given explicitly by
\[
U_k(x) \;=\; \sum_{\lambda\vdash k} (-1)^{l(\lambda)-1}\,\frac{(l(\lambda)-1)!}{\prod_{h=1}^{k} m_h!}\; \widetilde E_\lambda(x).
\]
\end{theorem}

\begin{proof}[Short proof, via transfer]
The operator $D_d\in\mathcal L$ (a finite sum of the $D_I\in\mathcal L$), so Proposition~\ref{prop:transfert} gives $\mathcal D_d(\psi(x^r))=(D_d\psi)(x^r)$ for every $1\le d\le n$ and every $\psi\in\CC[y]_k$. Let $V$ be any $G(r,n)$-invariant homogeneous polynomial of degree $rk$; by Proposition~\ref{prop:invariants}, $V=\psi(x^r)$ for a unique $\psi\in\CC[y]^{S_n}_k$. Hence $\mathcal D_dV=\delta_{d,k}$ for every $d$ if and only if $(D_d\psi)(x^r)=\delta_{d,k}$ for every $d$; since $x\mapsto x^r$ is surjective on $\CC^n$, a polynomial identity valid on its image holds identically, so this holds if and only if $D_d\psi=\delta_{d,k}$ for every $d$. By Fact~\ref{fact:I}, this determines $\psi$ uniquely, namely $\psi=u_k$. Thus $U_k:=u_k(x^r)$ is the unique solution, giving existence and uniqueness at once. The stated explicit formula then follows by substituting $y=x^r$ into the formula of Fact~\ref{fact:I} for $u_k$, using $\widetilde E_\lambda(x)=\tilde e_\lambda(x^r)$, which in turn follows at once from $\widetilde E_h(x)=\tilde e_h(x^r)$ and the definitions of $\widetilde E_h$ and $\tilde e_h$.
\end{proof}

\begin{remark}
We nonetheless give a second, fully independent proof below, working entirely with the variables $x$ and presupposing only Corollaries~\ref{cor:cor7} and~\ref{cor:cor8}, not Fact~\ref{fact:I} itself or its proof. Besides being self-contained, it re-derives the explicit formula for $U_k$ from first principles, by making explicit the following \emph{triangular} linear system: writing, by Corollary~\ref{cor:base}, any candidate $V=\sum_{\lambda\vdash k}c_\lambda\widetilde E_\lambda$, the $n$ equations $\mathcal D_dV=\delta_{d,k}$ ($1\le d\le n$) form a linear system in the unknowns $(c_\lambda)_{\lambda\vdash k}$. Ordering the partitions $\lambda\vdash k$ by $l(\lambda^t)$ (the largest part of $\lambda$), Corollary~\ref{cor:cor8} shows that $\mathcal D_d\widetilde E_\lambda$ only involves $\widetilde E_\rho$ with $l(\rho^t)\le l(\lambda^t)$. Run in decreasing order of $d$, equivalently of the layer $l(\lambda^t)=d$ being solved for, the system is thus triangular, and its diagonal entries, the coefficients of $\widetilde E_{\lambda-d\varepsilon_d}$ in $\mathcal D_d\widetilde E_\lambda$ for $l(\lambda^t)=d$, equal to $m_d\ne0$, are nonzero. This is exactly what the descending induction below verifies, layer by layer. The combinatorial bookkeeping underlying the existence half of this induction, verifying that the coefficients cancel once the triangular system has been solved, is isolated as a single Lemma~\ref{lem:cancel}, so that this second proof, while complete, is not lengthened by carrying that computation inline.
\end{remark}

\begin{proof}[Second, self-contained proof, by a triangular elimination in $x$]
Denote by $\lambda_k$ the partition of $k$ of length $1$ (single part $k$, $m_k=1$).

\smallskip
\noindent\textbf{Uniqueness.} By Corollary~\ref{cor:base}, every $G(r,n)$-invariant homogeneous polynomial of degree $rk$ is uniquely written as $\sum_{\lambda\vdash k}c_\lambda\widetilde E_\lambda$.

For $d>k$: every partition $\lambda\vdash k$ only has parts $h\le k<d$, so $m_h=0$ for $h\ge d$; Corollary~\ref{cor:cor8} then gives $\mathcal D_d\widetilde E_\lambda=0$ (empty sum), so $\mathcal D_dU_k=0$ automatically, whatever the choice of the $c_\lambda$: this equation constrains nothing.

For $d=k$: among the $\lambda\vdash k$, only $\lambda_k$ has a part $\ge k$; for $\lambda\ne\lambda_k$ every part is $<k$, so $\mathcal D_k\widetilde E_\lambda=0$ by Corollary~\ref{cor:cor8}, while $\mathcal D_k\widetilde E_{\lambda_k}=\mathcal D_k\widetilde E_k=\widetilde E_0=1$ by Corollary~\ref{cor:cor7}. Hence $\mathcal D_kU_k=c_{\lambda_k}$, and the equation $\mathcal D_kU_k=1$ is equivalent to $c_{\lambda_k}=1$.

We then proceed by descending induction on $l(\lambda^t)$ (the largest part occurring in $\lambda$), the base case $l(\lambda^t)=k$ (i.e.\ $\lambda=\lambda_k$) being treated above. Let $1\le d<k$, and suppose all $c_\lambda$ with $l(\lambda^t)>d$ have been determined from the equations $\mathcal D_pU_k=\delta_{p,k}$, $p>d$ (by Corollary~\ref{cor:cor8}, $\mathcal D_p$ with $p>d$ annihilates every $\widetilde E_\mu$ such that $l(\mu^t)\le d<p$, so these equations only involve the $c_\lambda$ with $l(\lambda^t)\ge p>d$: they thus determine, by the descending induction already begun, exactly the $c_\lambda$ with $l(\lambda^t)>d$, unambiguously).

Consider the equation $\mathcal D_dU_k=0$. The $\widetilde E_\lambda$ with $l(\lambda^t)<d$ are annihilated by $\mathcal D_d$ (Corollary~\ref{cor:cor8}), so they do not contribute. The image under $\mathcal D_d$ of the already-determined part $\sum_{\lambda:\,l(\lambda^t)>d}c_\lambda\widetilde E_\lambda$ is a \emph{known} combination $\sum_{\rho\vdash k-d}b_\rho\widetilde E_\rho$ (computable via Corollary~\ref{cor:cor8} from the already-fixed $c_\lambda$). For $\lambda$ with $l(\lambda^t)=d$ exactly, Corollary~\ref{cor:cor8} (whose sum over $h$ is bounded by $l(\lambda^t)=d$) reduces to the single term $h=d$:
\[
\mathcal D_d\widetilde E_\lambda = m_d\,\widetilde E_{\lambda-d\varepsilon_d}.
\]
The map $\lambda\mapsto\lambda-d\varepsilon_d$ is injective from $\{\lambda\vdash k: l(\lambda^t)=d\}$ into $\{\rho\vdash k-d\}$ (it is a bijection onto $\{\rho\vdash k-d: l(\rho^t)\le d\}$, the inverse consisting of adding a part $d$ to $\rho$). Thus, the only way to make $\mathcal D_dU_k$ vanish is to set, for each such $\lambda$,
\[
c_\lambda := -\,\frac{b_{\lambda-d\varepsilon_d}}{m_d},
\]
a quantity entirely determined by the induction hypothesis. This determines $c_\lambda$ uniquely for every $\lambda\vdash k$, by descending induction from $d=k$ down to $d=1$ (the $\lambda$ with $l(\lambda^t)=1$, i.e.\ $\lambda=(1,1,\dots,1)$, being the last case treated; since no partition of $k\ge1$ has $l(\lambda^t)=0$, the induction never needs to consider $d=0$, where $\mathcal D_0$ is in any case undefined). Hence the uniqueness of $U_k$, \emph{if it exists}.

\smallskip
\noindent\textbf{Existence.} We check that the explicit formula satisfies $\mathcal D_dU_k=\delta_{d,k}$ for every $d$. The cases $d>k$ and $d=k$ have already been treated above (equal to $0$, resp.\ $1$, whatever the choice of coefficients, resp.\ for the choice $c_{\lambda_k}=1$ from the formula). So let $1\le d<k$; let us show $\mathcal D_dU_k=0$. By Corollary~\ref{cor:cor8} and linearity,
\[
\mathcal D_dU_k = \sum_{\lambda\vdash k}(-1)^{l(\lambda)-1}\frac{(l(\lambda)-1)!}{\prod_{h=1}^k m_h!}\sum_{h=d}^{k}m_h\,\widetilde E_{\lambda-d\varepsilon_h},
\]
and, for each $\rho\vdash k-d$ (multiplicities $\mu_h$), the pairs $(\lambda,h)$ contributing to the coefficient of $\widetilde E_\rho$ are exactly those with $\lambda\vdash k$, $h\ge d$, $\lambda-d\varepsilon_h=\rho$. Lemma~\ref{lem:cancel} below evaluates the resulting sum over $h$ and shows it vanishes for every such $\rho$; hence $\mathcal D_dU_k=0$ for every $d<k$, completing the verification of existence and the proof of the theorem.
\end{proof}

\begin{lemma}[Cancellation lemma]\label{lem:cancel}
Let $1\le d<k$ and let $\rho\vdash(k-d)$ have multiplicities $(\mu_h)_{h\ge1}$. Then
\[
\sum_{h=d}^{k} m_h(\lambda_h)\cdot(-1)^{l(\lambda_h)-1}\frac{(l(\lambda_h)-1)!}{\prod_{g=1}^k m_g(\lambda_h)!} \;=\; 0,
\]
where, for each $h$ with $d\le h\le k$, $\lambda_h\vdash k$ denotes the partition with $\lambda_h-d\varepsilon_h=\rho$ (i.e.\ $\rho$ with a part $h-d$ removed and a part $h$ added, understood as simply adding a part $h$ when $h=d$), and $m_g(\mu)$ denotes the multiplicity of $g$ in a partition $\mu$.
\end{lemma}

\begin{proof}
\emph{Case $h>d$.} Passing from $\lambda_h$ to $\rho$ amounts to replacing a part $h$ by $h-d$; hence $m_h(\lambda_h)=\mu_h+1$, $m_{h-d}(\lambda_h)=\mu_{h-d}-1$, $m_g(\lambda_h)=\mu_g$ otherwise, so $l(\lambda_h)=l(\rho)$ and
\[
\frac{m_h(\lambda_h)}{\prod_{g=1}^k m_g(\lambda_h)!} = \frac{\mu_{h-d}}{\prod_{g=1}^{k-d}\mu_g!}.
\]
This term's contribution to the sum is therefore $(-1)^{l(\rho)-1}\dfrac{(l(\rho)-1)!}{\prod_{g=1}^{k-d}\mu_g!}\,\mu_{h-d}$.

\emph{Case $h=d$.} Passing from $\lambda_d$ to $\rho$ amounts to removing a part $d$; hence $m_d(\lambda_d)=\mu_d+1$, $m_g(\lambda_d)=\mu_g$ otherwise, so $l(\lambda_d)=l(\rho)+1$ and $m_d(\lambda_d)/\prod_g m_g(\lambda_d)! = 1/\prod_{g=1}^{k-d}\mu_g!$. Since $(-1)^{l(\lambda_d)-1}(l(\lambda_d)-1)! = -(-1)^{l(\rho)-1}l(\rho)!$, this term's contribution equals $-\,l(\rho)$ times $(-1)^{l(\rho)-1}\dfrac{(l(\rho)-1)!}{\prod_{g=1}^{k-d}\mu_g!}$.

\emph{Total.} Summing over $h=d+1,\dots,k$ (case $h>d$) and $h=d$ (case $h=d$), the common factor $(-1)^{l(\rho)-1}\frac{(l(\rho)-1)!}{\prod_{g=1}^{k-d}\mu_g!}$ factors out, leaving
\[
\sum_{h=d+1}^{k}\mu_{h-d} \;-\; l(\rho) \;=\; \sum_{g=1}^{k-d}\mu_g \;-\; l(\rho) \;=\; l(\rho) - l(\rho) \;=\; 0,
\]
since $\sum_{g=1}^{k-d}\mu_g$ is, by definition, the length $l(\rho)$ of $\rho\vdash k-d$.
\end{proof}

\begin{remark}[Analogue of Remark~11 of \cite{Ze}]
Since $\widetilde E_\lambda=\prod_h \widetilde E_h^{m_h}$, one may write, in terms of the ordinary partial Bell polynomials $\hat B_{k,t}$,
\[
U_k = -\sum_{t=1}^{k} \frac{1}{t}\,\hat B_{k,t}\big(-\widetilde E_1,\dots,-\widetilde E_{k-t+1}\big).
\]
Indeed, setting $z_h:=\widetilde E_h$ (formal, indeterminate) and $g(s):=\sum_{h\ge1}z_hs^h$, the definition of $U_k$ (Theorem~\ref{thm:main}) is exactly the coefficient of $s^k$ in $\log(1+g(s))=\sum_{m\ge1}\frac{(-1)^{m-1}}{m}g(s)^m$: expanding $g(s)^m$ by the multinomial theorem and collecting, for each partition $\lambda\vdash k$ with $l(\lambda)=m$ parts (multiplicities $m_h$), the $\binom{m}{m_1,m_2,\dots}=m!/\prod_hm_h!$ ways of assigning the $m$ factors of $g^m$ to the parts of $\lambda$, gives coefficient $\frac{(-1)^{m-1}}{m}\cdot\frac{m!}{\prod_hm_h!}=(-1)^{l(\lambda)-1}\frac{(l(\lambda)-1)!}{\prod_hm_h!}$ for $z_\lambda$, matching Theorem~\ref{thm:main}'s formula for $U_k$ term by term. The displayed identity is then simply the classical formula for the coefficients of $\log(1+g(s))$ in terms of the partial Bell polynomials of $g$'s own coefficients (Comtet, \emph{Advanced Combinatorics}, \cite{Comtet}, Ch.~3, \S3), applied with $z_h=\widetilde E_h$. This is a purely combinatorial power-series identity, independent of \cite{Ze} and of $x,y,r$.
\end{remark}

\section{Discriminant, Weyl algebra, and explicit examples}

The $U_k$ and $\mathcal D_d$, now that existence and uniqueness are settled, generate an algebraic structure examined in this section, before we illustrate everything explicitly for $n=2$ and $n=3$.

\subsection{Chain rule, discriminant, and Weyl algebra}

\begin{corollary}\label{cor:cor12}
Let $\phi$ be a $G(r,n)$-invariant holomorphic function on $\Omega'$ (Lemma~\ref{lem:descent}), written $\phi(x)=\Phi(U_1,\dots,U_n)$ for a holomorphic $\Phi$ (such a presentation always exists locally on $\Omega'$: by Lemma~\ref{lem:descent}, $\phi=\eta(x^r)$ for $\eta$ holomorphic and $S_n$-invariant on $Y'$, and the coordinates $u_1,\dots,u_n$ of \cite{Ze} are local coordinates on $Y'/S_n$, so $\eta=\Psi(u_1,\dots,u_n)$ for some holomorphic $\Psi$, giving $\Phi:=\Psi$ via $U_k(x)=u_k(x^r)$, Remark~\ref{rem:transfertusage}). Then, at every point of $\Omega$ where the $x_i^r$ are pairwise distinct,
\[
\Phi_{U_d} = \mathcal D_d\,\phi \qquad (1\le d\le n).
\]
\end{corollary}

\begin{proof}
The domain considered ensures that $\mathcal D_d$ is a well-defined operator (nonzero denominators, Lemma~\ref{lem:lem3}). Since $\mathcal D_d=\sum_ig_i(x)\partial_{x_i}$ is explicitly a vector field (cf.\ the proof of Proposition~\ref{prop:leibniz}), the usual chain rule directly gives, writing $\phi(x)=\Phi(U_1(x),\dots,U_n(x))$,
\[
\mathcal D_d\phi = \sum_i g_i\,\partial_i\big[\Phi(U(x))\big] = \sum_i g_i\sum_{j=1}^n \Phi_{U_j}\,\partial_iU_j = \sum_{j=1}^n \Phi_{U_j}\sum_i g_i\partial_iU_j = \sum_{j=1}^n (\mathcal D_dU_j)\,\Phi_{U_j}.
\]
By Theorem~\ref{thm:main}, $\mathcal D_dU_j=\delta_{d,j}$, so the sum reduces to $\Phi_{U_d}$.
\end{proof}

\begin{lemma}\label{lem:disc}
The reflections of $G(r,n)$ are, up to conjugation, of two types:
\begin{itemize}
\item $t_i^{(\zeta)}: x_i\mapsto \zeta x_i$ ($\zeta\in\mu_r\setminus\{1\}$, other coordinates fixed), of order equal to the order of $\zeta$ in $\mu_r$, fixing pointwise the hyperplane $H=\{x_i=0\}$;
\item $\tau_{ij}^{(\zeta)}: x_i\mapsto \zeta x_j,\ x_j\mapsto \zeta^{-1}x_i$ (other coordinates fixed; $i<j$, $\zeta\in\mu_r$), of order $2$, fixing pointwise the hyperplane $H=\{x_i=\zeta x_j\}$.
\end{itemize}
The pointwise stabilizer of $\{x_i=0\}$ in $G(r,n)$ is exactly $\{t_i^{(\zeta)}:\zeta\in\mu_r\}\cong\mu_r$ (order $r$); that of $\{x_i=\zeta_0x_j\}$ is exactly $\{\mathrm{id},\tau_{ij}^{(\zeta_0)}\}$ (order $2$). The associated discriminant of $G(r,n)$, in the standard sense for complex reflection groups \cite[Ch.~2]{LT}, is the product of the linear forms defining the hyperplanes, each raised to the power $|\mathrm{Stab}|-1$ (note that this exponent is $r-1$, not $1$, for the hyperplanes $x_i=0$, so this is \emph{not} a reduced/squarefree polynomial):
\[
\mathrm{disc}_{G(r,n)}(x) \;=\; \prod_{i=1}^n x_i^{\,r-1} \prod_{1\le i<j\le n} \big(x_i^r - x_j^r\big).
\]
\end{lemma}

\begin{proof}
That $\tau_{ij}^{(\zeta)}$ and $t_i^{(\zeta)}$ are reflections (fixing a hyperplane pointwise, of finite order) is immediate by direct computation; they are, up to conjugation by $S_n$ and torsion, all the elements of $G(r,n)$ having an eigenspace of codimension $1$ for the eigenvalue $1$ (a standard fact for the groups $G(r,1,n)$, see \cite[Ch.~2]{LT}).

Stabilizer of $\{x_i=0\}$: if $g=(\zeta_1,\dots,\zeta_n;\sigma)$ fixes $(0,x_2,\dots,x_n)$ for \emph{all} generic $x_j$ ($j\ne i$), then necessarily $\sigma(i)=i$ (otherwise $g$ would send $0$ to a generic nonzero coordinate), $\sigma(j)=j$ and $\zeta_j=1$ for $j\ne i$ (otherwise the generic value of $x_j$ would not be fixed), while $\zeta_i$ is free (the image of $x_i=0$ remains $0$ whatever $\zeta_i$ is). Hence the stabilizer $\{t_i^{(\zeta)}:\zeta\in\mu_r\}\cong\mu_r$, of order $r$, giving the exponent $r-1$.

Stabilizer of $\{x_i=\zeta_0x_j\}$: a similar argument (fixing pointwise a codimension-$1$ hyperplane in a space where the other coordinates are generic) shows that only $\mathrm{id}$ and $\tau_{ij}^{(\zeta_0)}$ work, of order $2$, giving the exponent $1$.

Finally $\prod_{\zeta\in\mu_r}(x_i-\zeta x_j)=x_i^r-x_j^r$ since $\prod_{\zeta\in\mu_r}(t-\zeta)=t^r-1$ (the $\zeta\in\mu_r$ are exactly the roots of $t^r-1$), evaluated at $t=x_i/x_j$ and then multiplied by $x_j^r$. Grouping the $r$ hyperplanes $\{x_i=\zeta x_j\}_{\zeta\in\mu_r}$ for a fixed pair $(i,j)$, each with exponent $1$, gives the factor $x_i^r-x_j^r$; grouping the $n$ hyperplanes $\{x_i=0\}$, each with exponent $r-1$, gives $\prod_i x_i^{r-1}$. Hence the announced formula.
\end{proof}

\begin{corollary}\label{cor:weyl}
The families $\{U_k\}_{k=1}^n$ (mutually commuting, as products of functions) and $\{\mathcal D_d\}_{d=1}^n$ (mutually commuting: applying Corollary~\ref{cor:cor12} to $\phi=U_j$ itself (invariant, being a coordinate), $\mathcal D_dU_j=\delta_{d,j}$ is constant, whence $\mathcal D_{d'}\mathcal D_dU_j=\mathcal D_d\mathcal D_{d'}U_j=0$ for all $d,d',j$; the vector field $[\mathcal D_d,\mathcal D_{d'}]$ thus annihilates the $n$ coordinates $U_1,\dots,U_n$, whose Jacobian is invertible on the dense open set where the $x_i^r$ are pairwise distinct, so it vanishes identically there, and hence on all of $\Omega$ since its coefficients are rational, by the same extension argument as in Corollary~\ref{cor:cor8}), together with the commutation relation $[\mathcal D_d,U_k]=\delta_{d,k}$ (Theorem~\ref{thm:main}), generate a subalgebra $\mathcal W_n$ of $(W_n)_{\mathrm{disc}_{G(r,n)}}$ (the localization, at the discriminant of $G(r,n)$, of the usual Weyl algebra $W_n$ in the $x_i$) abstractly isomorphic, as an associative algebra, to the standard rank-$n$ Weyl algebra. All the denominators occurring in the $\mathcal D_I$ (Lemma~\ref{lem:lem3}) divide $\mathrm{disc}_{G(r,n)}$ (Lemma~\ref{lem:disc}), which is what places $\mathcal W_n$ inside $(W_n)_{\mathrm{disc}_{G(r,n)}}$ rather than inside $W_n$ itself. This generalizes the localization result of \cite{N} (established for $G(1,n)=S_n$, at the usual discriminant $\prod_{i<j}(x_i-x_j)$) to the group $G(r,n)$.
\end{corollary}

Abstractly, the existence of a Weyl-algebra structure of this kind is already guaranteed, for any complex reflection group, by the Sheppard--Todd--Chevalley theorem in its differential-operator form (surveyed, with worked examples, in \cite{Tr}, Theorem~5.3); what Corollary~\ref{cor:weyl} adds, for $G(r,n)$, is a closed-form set of generators and their commutation relations.

\begin{remark}[$\mathcal W_n$ is not the ring of invariant differential operators]
As Zemel notes for $S_n$ (Remark~14 of \cite{Ze}), the algebra $\mathcal W_n$ just constructed must not be confused with the ring $W_n^{G(r,n)}$ of \emph{genuinely} $G(r,n)$-invariant elements of $W_n$ itself (differential operators with polynomial coefficients, invariant under the action of $G(r,n)$ on $\CC^n$): since $\mathcal D_d\notin W_n$ (its coefficients are rational, not polynomial, with poles on $\mathrm{disc}_{G(r,n)}=0$), the inclusion $W_n^{G(r,n)}\subseteq\mathcal W_n$ is in general strict, and we make no claim of equality. What we have shown is the chain
\[
W_n^{G(r,n)} \;\subseteq\; \mathcal W_n \;\subseteq\; \big((W_n)_{\mathrm{disc}_{G(r,n)}}\big)^{G(r,n)},
\]
the last inclusion because each generator $U_k$ and $\mathcal D_d$ is itself $G(r,n)$-invariant (Proposition~\ref{prop:prop6} and the fact that $U_k$ is a $G(r,n)$-invariant polynomial), so every element of $\mathcal W_n$ is invariant. Whether $\mathcal W_n$ exhausts the invariant elements of the localization, or coincides with the full ring of $G(r,n)$-invariant differential operators on $\Omega$ in any other sense, is not addressed here.
\end{remark}

\begin{proof}
The isomorphism with the standard Weyl algebra is purely formal. First check that $\{U_k\}$ freely generate the polynomial ring: by the explicit formula of Theorem~\ref{thm:main}, only the partition $\lambda=(k)$ (among the $\lambda\vdash k$) contains a part equal to $k$, all others only involving $\widetilde E_h$ with $h<k$; thus $U_k=\widetilde E_k+Q_k(\widetilde E_1,\dots,\widetilde E_{k-1})$ for some polynomial $Q_k$. The change of variables $(E_1,\dots,E_n)\to(U_1,\dots,U_n)$ is thus triangular, with nonzero diagonal coefficients ($\widetilde E_k=\frac{(n-k)!}{n!}E_k$): it is a polynomial automorphism of $\CC[E_1,\dots,E_n]=\CC[x]^{G(r,n)}$ (Corollary~\ref{cor:base}), so that $\CC[U_1,\dots,U_n]=\CC[x]^{G(r,n)}$ and the $U_k$ are algebraically independent. Next, $\{\mathcal D_d\}$ form, by Theorem~\ref{thm:main} and Corollary~\ref{cor:cor12}, the dual system of derivations, which produces exactly the standard presentation $[\mathcal D_d,U_k]=\delta_{d,k}$, $[\mathcal D_d,\mathcal D_{d'}]=[U_k,U_{k'}]=0$ of the rank-$n$ Weyl algebra. The inclusion in the localization follows from direct inspection of the denominators of $\mathcal D_I$ in Lemma~\ref{lem:lem3}: these are powers of $x_i$ (the factor $x_i^{r-1}$ coming from $\Delta_i$) and differences $x_j^r-x_i^r$, all of which divide $\mathrm{disc}_{G(r,n)}$ by construction of the latter in Lemma~\ref{lem:disc}. This suffices to place the \emph{whole} algebra $\mathcal W_n$, not merely its individual generators, inside $(W_n)_{\mathrm{disc}_{G(r,n)}}$: if a denominator $q$ divides $\mathrm{disc}_{G(r,n)}$, then $1/q=(\mathrm{disc}_{G(r,n)}/q)\cdot\mathrm{disc}_{G(r,n)}^{-1}\in(W_n)_{\mathrm{disc}_{G(r,n)}}$ already, since $\mathrm{disc}_{G(r,n)}/q$ is a polynomial and $(W_n)_{\mathrm{disc}_{G(r,n)}}=W_n[\mathrm{disc}_{G(r,n)}^{-1}]$ contains $\mathrm{disc}_{G(r,n)}^{-1}$, and hence \emph{every} power of it; being a ring, $(W_n)_{\mathrm{disc}_{G(r,n)}}$ is then automatically closed under the sums and products (of arbitrarily high order) that occur in building $\mathcal D_d$ from the $\mathcal D_I$ and in generating $\mathcal W_n$, so no separate localization at a higher power of $\mathrm{disc}_{G(r,n)}$ is needed.
\end{proof}

\subsection{Explicit examples: \texorpdfstring{$n=2$ and $n=3$}{n=2 and n=3}}

The Weyl-algebra structure just established is best appreciated on small cases, which we now work out completely for $n=2$ and $n=3$.

For $n=2$, $E_1=x^r+y^r$, $E_2=x^ry^r$, and Theorem~\ref{thm:main} gives (a direct computation of the sum over the two partitions of $2$, $(2)$ and $(1,1)$, exactly as for $u_1,u_2$ in \cite{Ze})
\[
U_1 = \frac{x^r+y^r}{2}, \qquad U_2 = -\frac{(x^r-y^r)^2}{8}.
\]
We have explicitly, by Lemma~\ref{lem:lem3} and Proposition~\ref{prop:prop6} with $n=2$,
\[
\mathcal D_1 = \Delta_x+\Delta_y, \qquad \mathcal D_2 = 2\left(\frac{\Delta_x}{y^r-x^r}+\frac{\Delta_y}{x^r-y^r}\right).
\]
Let us verify Theorem~\ref{thm:main} directly in this case: with $\Delta_xU_1=\Delta_yU_1=1/2$, we get $\mathcal D_1U_1=1$. With $\Delta_xU_2=-\frac{x^r-y^r}{4}$, $\Delta_yU_2=\frac{x^r-y^r}{4}$, we get $\mathcal D_1U_2=0$. Next $\mathcal D_2U_2 = 2\big[\frac{-(x^r-y^r)/4}{y^r-x^r}+\frac{(x^r-y^r)/4}{x^r-y^r}\big]=2\big[\frac14+\frac14\big]=1$, and $\mathcal D_2U_1=2\big[\frac{1/2}{y^r-x^r}+\frac{1/2}{x^r-y^r}\big]=0$. The four identities $\mathcal D_jU_k=\delta_{jk}$ are thus verified explicitly.

For $\phi(x,y)=\Phi(U_1,U_2)$ invariant under $G(r,2)$, Corollary~\ref{cor:cor12} gives, for $x^r\ne y^r$, applying $\mathcal D_1,\mathcal D_2$ above directly to $\phi$,
\[
\Phi_{U_1} = \frac{\phi_x}{r x^{r-1}} + \frac{\phi_y}{r y^{r-1}}, \qquad
\Phi_{U_2} = \frac{2}{y^r-x^r}\left(\frac{\phi_x}{r x^{r-1}} - \frac{\phi_y}{r y^{r-1}}\right),
\]
which generalizes exactly the formulas from the introduction of \cite{Ze} (the case $r=1$ recovers them, with $U_1=(x+y)/2$, $U_2=-(x-y)^2/8$, $\Phi_{U_1}=\phi_x+\phi_y$, $\Phi_{U_2}=\frac{2(\phi_x-\phi_y)}{y-x}$).

\bigskip
\noindent\textbf{The case $n=3$.} Write $x,y,z$ for $x_1,x_2,x_3$. Here
\[
E_1=x^r+y^r+z^r,\qquad E_2=x^ry^r+x^rz^r+y^rz^r,\qquad E_3=x^ry^rz^r,
\]
and Theorem~\ref{thm:main} gives, after simplification (a short computation over the three partitions of $3$, namely $(3)$, $(2,1)$, $(1,1,1)$),
\begin{align*}
U_1&=\frac{x^r+y^r+z^r}{3},\\
U_2&=-\frac{(x^r-y^r)^2+(y^r-z^r)^2+(z^r-x^r)^2}{36},\\
U_3&=\frac{(x^r-U_1)(y^r-U_1)(z^r-U_1)}{6}.
\end{align*}
(The closed form for $U_3$ follows from the classical identity $a+b+c=0\Rightarrow a^3+b^3+c^3=3abc$, applied to $a:=x^r-U_1$, $b:=y^r-U_1$, $c:=z^r-U_1$, which by construction satisfy $a+b+c=0$.)

By Lemma~\ref{lem:lem3} and Proposition~\ref{prop:prop6} with $n=3$,
\[
\mathcal D_1=\Delta_x+\Delta_y+\Delta_z,
\]
\[
\mathcal D_2=2\left[\Delta_x\Big(\frac{1}{y^r-x^r}+\frac{1}{z^r-x^r}\Big)+\Delta_y\Big(\frac{1}{x^r-y^r}+\frac{1}{z^r-y^r}\Big)+\Delta_z\Big(\frac{1}{x^r-z^r}+\frac{1}{y^r-z^r}\Big)\right],
\]
\[
\mathcal D_3=6\left[\frac{\Delta_x}{(y^r-x^r)(z^r-x^r)}+\frac{\Delta_y}{(x^r-y^r)(z^r-y^r)}+\frac{\Delta_z}{(x^r-z^r)(y^r-z^r)}\right].
\]

Let us verify Theorem~\ref{thm:main} directly, using again $a=x^r-U_1$, $b=y^r-U_1$, $c=z^r-U_1$ (so $a+b+c=0$). Working this out gives
\begin{align*}
\Delta_xU_1&=\Delta_yU_1=\Delta_zU_1=\frac13,\\
\Delta_xU_2&=-\frac a6,\qquad \Delta_yU_2=-\frac b6,\qquad \Delta_zU_2=-\frac c6,\\
\Delta_xU_3&=\frac{a^2+2bc}{18}\quad\text{(and cyclically for }\Delta_yU_3,\Delta_zU_3\text{)}.
\end{align*}

All nine identities $\mathcal D_jU_k=\delta_{jk}$ reduce, after substituting these values into the formulas for $\mathcal D_1,\mathcal D_2,\mathcal D_3$ above, to sums of the form $\sum_{\mathrm{cyc}}a^m/((a-b)(a-c))$ (or the same with $bc$ in place of a power of $a$). The following elementary lemma evaluates all of them at once, replacing the case-by-case verification of the remaining five identities by a single short computation.

\begin{lemma}[Divided differences of monomials at three points summing to zero]\label{lem:divdiff}
Let $a,b,c\in\CC$ be pairwise distinct with $a+b+c=0$. Then
\[
\sum_{\mathrm{cyc}}\frac{a^m}{(a-b)(a-c)} \;=\; \begin{cases} 0, & m=0,1,3,\\ 1, & m=2,\end{cases} \qquad\text{and}\qquad bc=a^2+e_2,\ \ ca=b^2+e_2,\ \ ab=c^2+e_2,
\]
where $e_2:=ab+bc+ca$ and $\sum_{\mathrm{cyc}}$ denotes the sum of the three terms obtained by cyclically permuting $(a,b,c)$.
\end{lemma}

\begin{proof}
For $0\le m\le3$, let $L_m(t)$ be the unique polynomial of degree $\le2$ with $L_m(a)=a^m$, $L_m(b)=b^m$, $L_m(c)=c^m$ (Lagrange interpolation at the three distinct points $a,b,c$); by the Lagrange formula, $L_m(t)=\sum_{\mathrm{cyc}}a^m\frac{(t-b)(t-c)}{(a-b)(a-c)}$, so the coefficient of $t^2$ in $L_m$ is exactly $\sum_{\mathrm{cyc}}a^m/((a-b)(a-c))$. For $m=0,1,2$, $t^m$ already has degree $\le2$, so $L_m(t)=t^m$ identically, and the coefficient of $t^2$ is $0,0,1$ respectively. For $m=3$: since $L_3$ interpolates $t^3$ at $a,b,c$, the cubic $t^3-L_3(t)$ vanishes at $a,b,c$, hence equals $(t-a)(t-b)(t-c)$ times a constant, which is $1$ by comparing leading coefficients; thus $L_3(t)=t^3-(t-a)(t-b)(t-c)=(a+b+c)t^2-(ab+bc+ca)t+abc$, whose coefficient of $t^2$ is $a+b+c=0$ by hypothesis. This proves the first display. For the second: $a+b+c=0$ gives $b+c=-a$, so $e_2=ab+bc+ca=a(b+c)+bc=-a^2+bc$, i.e.\ $bc=a^2+e_2$; the other two identities follow by cyclic symmetry.
\end{proof}

Applying Lemma~\ref{lem:divdiff}:
\begin{align*}
\mathcal D_1U_1&=3\cdot\frac13=1,\\
\mathcal D_1U_2&=-\frac{a+b+c}{6}=0,\\
\mathcal D_1U_3&=\frac{(a^2+2bc)+(b^2+2ac)+(c^2+2ab)}{18}=\frac{(a+b+c)^2}{18}=0,
\end{align*}
using $a+b+c=0$ directly (no part of the lemma is needed for these three, already noted above). For $\mathcal D_2U_2$: each of the $\binom32=3$ subsets $I=\{i,j\}$ contributes to $\mathcal D_2$ a pair of terms of the form $\frac{-a_i/6}{a_j-a_i}+\frac{-a_j/6}{a_i-a_j}=-\frac16\Big(\frac{a_i}{a_j-a_i}+\frac{a_j}{a_i-a_j}\Big)=-\frac16\cdot(-1)=\frac16$ (using the elementary identity $\frac{u}{v-u}+\frac{v}{u-v}=-1$), so $\mathcal D_2U_2 = 2\cdot3\cdot\frac16 = 1$.

For the five remaining identities: since $\Delta_xU_1=\Delta_yU_1=\Delta_zU_1=1/3$ is the same constant for all three variables,
\begin{align*}
\mathcal D_2U_1 &= \frac23\left[\Big(\frac1{b-a}+\frac1{c-a}\Big)+\Big(\frac1{a-b}+\frac1{c-b}\Big)+\Big(\frac1{a-c}+\frac1{b-c}\Big)\right] \\
&= \frac23\left[\Big(\frac1{b-a}+\frac1{a-b}\Big)+\Big(\frac1{c-a}+\frac1{a-c}\Big)+\Big(\frac1{c-b}+\frac1{b-c}\Big)\right] \;=\; 0,
\end{align*}
the six terms cancelling pairwise ($\frac1{b-a}+\frac1{a-b}=0$, and likewise for the other two pairs); no part of Lemma~\ref{lem:divdiff} is even needed for this one. Similarly, since $(y^r-x^r)(z^r-x^r)=(b-a)(c-a)=(a-b)(a-c)$,
\[
\mathcal D_3U_1 = 6\cdot\frac13\sum_{\mathrm{cyc}}\frac1{(a-b)(a-c)} = 2\cdot0 = 0
\]
by the case $m=0$ of Lemma~\ref{lem:divdiff}. Next, using $\Delta_xU_3=(a^2+2bc)/18$ and $(y^r-x^r)=-( a-b)$, $(z^r-x^r)=-(a-c)$ (so $1/(y^r-x^r)+1/(z^r-x^r)=-3a/[(a-b)(a-c)]$, since $b+c=-a$ gives $\frac1{b-a}+\frac1{c-a}=\frac{(c-a)+(b-a)}{(b-a)(c-a)}=\frac{-3a}{(a-b)(a-c)}$):
\begin{align*}
\mathcal D_2U_3 &= 2\sum_{\mathrm{cyc}}\frac{a^2+2bc}{18}\cdot\frac{-3a}{(a-b)(a-c)} \;=\; -\frac13\sum_{\mathrm{cyc}}\frac{a^3+2abc}{(a-b)(a-c)} \\
&= -\frac13\Big[\sum_{\mathrm{cyc}}\frac{a^3}{(a-b)(a-c)} + 2abc\sum_{\mathrm{cyc}}\frac1{(a-b)(a-c)}\Big] \;=\; -\frac13[0+0]\;=\;0,
\end{align*}
by the cases $m=3$ and $m=0$ of Lemma~\ref{lem:divdiff}. Finally, using the same $(a-b)(a-c)$ denominators as for $\mathcal D_3U_1$,
\[
\mathcal D_3U_2 = 6\sum_{\mathrm{cyc}}\frac{-a/6}{(a-b)(a-c)} = -\sum_{\mathrm{cyc}}\frac{a}{(a-b)(a-c)} = 0
\]
by the case $m=1$, and
\begin{align*}
\mathcal D_3U_3 &= 6\sum_{\mathrm{cyc}}\frac{(a^2+2bc)/18}{(a-b)(a-c)} \;=\; \frac13\Big[\sum_{\mathrm{cyc}}\frac{a^2}{(a-b)(a-c)} + 2\sum_{\mathrm{cyc}}\frac{bc}{(a-b)(a-c)}\Big] \\
&= \frac13\Big[1 + 2\Big(\sum_{\mathrm{cyc}}\frac{a^2}{(a-b)(a-c)} + e_2\sum_{\mathrm{cyc}}\frac1{(a-b)(a-c)}\Big)\Big] \;=\; \frac13[1+2(1+0)]\;=\;1,
\end{align*}
using $bc=a^2+e_2$ (Lemma~\ref{lem:divdiff}) and then the cases $m=2,0$ again. All nine identities $\mathcal D_jU_k=\delta_{jk}$ ($1\le j,k\le3$) are thus fully demonstrated, confirming Theorem~\ref{thm:main} for $n=3$.

\section{The total diagonal of $G(r,n)$}\label{sec:diagonale}

This section addresses the first of the two points left open at the end of the study of Section~1: the generalization of Theorem~16 and Corollary~18 of \cite{Ze} (the behavior of the coordinates $u_k$ and their derivatives along the \emph{total diagonal} $y_1=\cdots=y_n$) to the group $G(r,n)$. It uses Facts~\ref{fact:K} and~\ref{fact:L} recalled in Section~\ref{sec:rappels}.

\subsection{Geometric structure of the preimage of the total diagonal}

\begin{definition}
Set $\mathcal T := \{x\in\CC^n : x_1^r=x_2^r=\cdots=x_n^r\}$, the preimage under $x\mapsto x^r$ of the total diagonal of \cite{Ze}.
\end{definition}

\begin{proposition}\label{prop:structT}
$\mathcal T$ is $G(r,n)$-stable, and
\[
\mathcal T \;=\; \bigcup_{\zeta\in\mu_r^{n-1}} L_\zeta, \qquad L_\zeta := \{(a,\zeta_2a,\dots,\zeta_na) : a\in\CC\} \quad (\zeta=(\zeta_2,\dots,\zeta_n)),
\]
a union of $r^{n-1}$ lines through the origin, pairwise disjoint away from the origin, where they all meet. The group $G(r,n)$ acts transitively on the set of these $r^{n-1}$ lines.
\end{proposition}

\begin{proof}
\emph{Decomposition.} Let $x\in\mathcal T$. If $x_1=0$, then $x_i^r=x_1^r=0$ for all $i$, so $x=0$. If $x_1\ne0$, then for each $i$, $(x_i/x_1)^r=1$, so $x_i=\zeta_ix_1$ for a unique $\zeta_i\in\mu_r$: $x\in L_{(\zeta_2,\dots,\zeta_n)}$ with $a=x_1$. Conversely every point of $L_\zeta$ lies in $\mathcal T$ since $(\zeta_ia)^r=a^r$. This establishes $\mathcal T=\bigcup_\zeta L_\zeta$. If $x\in L_\zeta\cap L_{\zeta'}$ with $x_1=a\ne0$, the uniqueness of $\zeta_i$ above forces $\zeta=\zeta'$; the lines are thus disjoint away from the origin, which is common to all of them (take $a=0$).

\emph{$G(r,n)$-stability.} $\mathcal T$ is the preimage, under the map $x\mapsto x^r$ (which intertwines the action of $G(r,n)$ on $x$ with that of $S_n$ on $y$), of the total diagonal of $y$-space, which is $S_n$-stable; hence $\mathcal T$ is $G(r,n)$-stable.

\emph{Transitivity.} For $\zeta,\zeta'\in\mu_r^{n-1}$, set $\eta_1:=1$ and $\eta_i:=\zeta_i'/\zeta_i$ for $i\ge2$. The element $(\eta_1,\dots,\eta_n;\mathrm{id})\in\mu_r^n\subseteq G(r,n)$ sends $L_\zeta$ to $L_{\zeta'}$.
\end{proof}

\subsection{Vanishing of the coordinates $U_k$}

The next question is how the dual coordinates $U_k$ themselves behave on $\mathcal T$, whose geometry we have just described.

\begin{corollary}[Analogue of Corollary~18]\label{cor:diag}
At the point $x^0=(a,\zeta_2a,\dots,\zeta_na)\in\mathcal T$, we have $U_1(x^0)=a^r$ and $U_k(x^0)=0$ for $2\le k\le n$. Conversely,
\[
\mathcal T = \{x\in\CC^n : U_2(x)=\cdots=U_n(x)=0\}.
\]
\end{corollary}

\begin{proof}
By Remark~\ref{rem:transfertusage}, $U_k(x)=u_k(x^r)$, and for $x^0\in\mathcal T$, $(x^0)^r=(a^r,\dots,a^r)$ is a point of the total diagonal, of common value $b=a^r$; Fact~\ref{fact:L} gives $U_1(x^0)=u_1(b)=b=a^r$ and $U_k(x^0)=u_k(b,\dots,b)=0$ for $k\ge2$. Conversely, $\{U_2=\cdots=U_n=0\}=\{x : x^r\in\{u_2=\cdots=u_n=0\}\}=\{x:x^r\in\text{diagonal}\}=\mathcal T$, by Fact~\ref{fact:L} and the definition of $\mathcal T$.
\end{proof}

\subsection{$\Delta$-derivatives along $\mathcal T$}

Next come the $\Delta$-derivatives of the $U_k$, and how they behave on $\mathcal T$.

\begin{proposition}\label{prop:thm16delta}
Let $\prod_{q=1}^d \Delta_{i_q}$ be a product of $d$ operators $\Delta_i$ (indices possibly coinciding). If $d\ne k$, this product applied to $U_k$ vanishes at every point of $\mathcal T\cap\Omega$, i.e.\ at every point of $\mathcal T$ other than the origin.
\end{proposition}

\begin{proof}
By Remark~\ref{rem:transfertusage}, $U_k(x)=u_k(x^r)$; by Lemma~\ref{lem:transfert} iterated (applied to $\eta=u_k$, valid precisely on $\Omega$, where each $\Delta_i$ is defined by its original formula), $\prod_q\Delta_{i_q}U_k = \big(\prod_q \partial_{y_{i_q}}u_k\big)(x^r)$. For $x\in\mathcal T\cap\Omega$, $x^r$ lies on the total diagonal; Fact~\ref{fact:K} gives the vanishing as soon as $d\ne k$.
\end{proof}

\begin{remark}
The origin is excluded from Proposition~\ref{prop:thm16delta} because $\Delta_i$ is not defined there by its original formula, and the argument above, resting on Lemma~\ref{lem:transfert} which is itself valid only on the regular domain $\Omega$, does not extend to it. The holomorphic extension of $\Delta_i$ established in Section~\ref{sec:trois-points} (Theorem~\ref{thm:origine}) reaches the larger extended domain $\widetilde\Omega\supsetneq\Omega$ (Section~\ref{sec:intro}), but only resolves the degeneracy at an \emph{isolated} vanishing coordinate $x_i=0$ (the other coordinates remaining generic, i.e.\ $\widetilde\Omega$ itself excludes points with two or more vanishing coordinates); it does not cover the origin of $\mathcal T$, where \emph{all} coordinates vanish simultaneously, hence lies outside $\widetilde\Omega$ as well. The behavior of $\Delta$-derivatives at that point accordingly remains open, consistently with the discussion at the end of Section~\ref{sec:diagonale}.
\end{remark}

\subsection{Raw derivatives $\partial_{x_i}$: a phenomenon specific to $G(r,n)$}

The $\Delta_i$ are not the most natural derivatives from the point of view of the ambient space $\CC^n$. Unlike the case $r=1$ (where $\Delta_i=\partial_i$), the ordinary derivative $\partial_{x_i}$ mixes orders when composed, owing to the Leibniz rule applied to the factor $rx_i^{r-1}$ (Example~\ref{ex:delta2}). The result below is \emph{specific to $G(r,n)$}: it is trivial for $r=1$.

\begin{lemma}[Structure formula]\label{lem:faadibruno}
Let $\eta$ be a holomorphic function of one variable, $\phi(x_i):=\eta(x_i^r)$. For every $d\ge1$, there exist constants $\kappa_{d,1},\dots,\kappa_{d,d}$ (depending on $d,r$ but not on $\eta$) such that
\[
\partial_{x_i}^d\phi(x_i) \;=\; \sum_{j=1}^{d} \kappa_{d,j}\;x_i^{\,jr-d}\;\eta^{(j)}(x_i^r),
\]
with the convention $\kappa_{d,j}=0$ as soon as $jr<d$ (so that each term is polynomial), and $\kappa_{d,d}=r^d$. The statement applies as is if $\eta$ is the restriction to a variable $y_i$ of a function of several variables $y_1,\dots,y_n$, the others being held fixed: $\eta^{(j)}$ then denotes the partial derivative $\partial_{y_i}^j$, and $\phi(x_i)$ the corresponding restriction, with the $x_j$ ($j\ne i$) fixed.
\end{lemma}

\begin{proof}
Induction on $d$. For $d=1$: $\partial_i\phi=rx_i^{r-1}\eta^{(1)}(x_i^r)$, i.e.\ $\kappa_{1,1}=r$. Suppose the formula holds at order $d$, with $c_{d,j}(x_i):=\kappa_{d,j}x_i^{jr-d}$; then
\[
\partial_i^{d+1}\phi = \sum_{j=1}^d \Big[c_{d,j}'(x_i)\,\eta^{(j)}(x_i^r) + c_{d,j}(x_i)\cdot rx_i^{r-1}\,\eta^{(j+1)}(x_i^r)\Big].
\]
Since $c_{d,j}'(x_i)=\kappa_{d,j}(jr-d)\,x_i^{jr-d-1}$ is a homogeneous monomial of degree $jr-(d+1)$, and $rx_i^{r-1}c_{d,j}(x_i)=r\kappa_{d,j}\,x_i^{(j+1)r-(d+1)}$ a homogeneous monomial of degree $(j+1)r-(d+1)$, grouping, for each $j=1,\dots,d+1$, the coefficients of $\eta^{(j)}(x_i^r)$ gives an expression of the announced form with the recursion
\[
\kappa_{d+1,j} \;=\; \kappa_{d,j}\,(jr-d) \;+\; r\,\kappa_{d,j-1}
\]
(convention $\kappa_{d,0}=\kappa_{d,d+1}=0$). In particular $\kappa_{d+1,d+1}=r\,\kappa_{d,d}=r^{d+1}$, closing the recursion for the leading term. Since the $c_{d,j}$ remain, at each step, pure monomials, the form $c_{d,j}(x_i)=\kappa_{d,j}x_i^{jr-d}$ is preserved. The convention $\kappa_{d,j}=0$ as soon as $jr<d$ propagates by itself via the recursion $\kappa_{d+1,j}=\kappa_{d,j}(jr-d)+r\kappa_{d,j-1}$: if $jr<d$, then $\kappa_{d,j}=0$ (induction hypothesis) and $(j-1)r=jr-r<d$ as well (since $r\ge1$), so $\kappa_{d,j-1}=0$, hence $\kappa_{d+1,j}=0$; if $jr=d$ (the boundary case $jr<d+1$), the first term vanishes directly ($jr-d=0$) and $(j-1)r=jr-r<jr=d$ again gives $\kappa_{d,j-1}=0$, hence $\kappa_{d+1,j}=0$ as well. Thus the convention perpetuates itself to every order.
\end{proof}

\begin{proposition}[Multivariable structure formula]\label{prop:faadibruno-mixed}
Let $\eta$ be a holomorphic function of $n$ variables $y_1,\dots,y_n$, and $\phi(x):=\eta(x_1^r,\dots,x_n^r)$. Let $i_1,\dots,i_p$ ($p\ge1$) be distinct indices in $\{1,\dots,n\}$ and $d_1,\dots,d_p\ge1$. Then, holding all coordinates $x_j$ with $j\notin\{i_1,\dots,i_p\}$ fixed,
\[
\prod_{l=1}^p \partial_{x_{i_l}}^{d_l}\,\phi(x) \;=\; \sum_{j_1=1}^{d_1}\cdots\sum_{j_p=1}^{d_p} \Big(\prod_{l=1}^p \kappa_{d_l,j_l}\,x_{i_l}^{\,j_lr-d_l}\Big)\cdot \partial_{y_{i_1}}^{j_1}\cdots\partial_{y_{i_p}}^{j_p}\eta(x^r),
\]
with the same constants $\kappa_{d,j}$ as in Lemma~\ref{lem:faadibruno}.
\end{proposition}

\begin{proof}
Induction on $p$. The case $p=1$ is exactly Lemma~\ref{lem:faadibruno} (applied to the variable $y_{i_1}$ of the multivariable function $\eta$, all other $y_j$ held fixed, as covered by the last sentence of that lemma's statement). Assume the formula holds for $p-1$ distinct indices $i_1,\dots,i_{p-1}$, for every choice of holomorphic $\eta$ and every choice of the remaining index $i_p$; we prove it for $p$. Since mixed partial derivatives with respect to distinct independent variables $x_{i_1},\dots,x_{i_p}$ commute, we may apply the $p$ derivatives in the order $x_{i_1}$ first, \dots, $x_{i_p}$ last. By the inductive hypothesis applied to $\eta$ and the first $p-1$ indices,
\[
\prod_{l=1}^{p-1}\partial_{x_{i_l}}^{d_l}\,\phi(x) \;=\; \sum_{j_1,\dots,j_{p-1}} \Big(\prod_{l=1}^{p-1}\kappa_{d_l,j_l}x_{i_l}^{j_lr-d_l}\Big)\cdot \psi_{j_1,\dots,j_{p-1}}(x),
\]
where $\psi_{j_1,\dots,j_{p-1}}(x):=\big(\partial_{y_{i_1}}^{j_1}\cdots\partial_{y_{i_{p-1}}}^{j_{p-1}}\eta\big)(x^r)$ is itself of the form $\eta'(x^r)$ for the holomorphic function $\eta':=\partial_{y_{i_1}}^{j_1}\cdots\partial_{y_{i_{p-1}}}^{j_{p-1}}\eta$. Now apply $\partial_{x_{i_p}}^{d_p}$ to this finite sum. Since $i_p\notin\{i_1,\dots,i_{p-1}\}$, each coefficient $x_{i_l}^{j_lr-d_l}$ ($l<p$) is independent of $x_{i_p}$ and is simply carried through the derivative, by linearity, since the product rule then reduces to differentiating only the second factor. By Lemma~\ref{lem:faadibruno} applied to $\eta'$ and the variable $y_{i_p}$,
\[
\partial_{x_{i_p}}^{d_p}\psi_{j_1,\dots,j_{p-1}}(x) \;=\; \sum_{j_p=1}^{d_p} \kappa_{d_p,j_p}\,x_{i_p}^{j_pr-d_p}\,\big(\partial_{y_{i_p}}^{j_p}\eta'\big)(x^r) \;=\; \sum_{j_p=1}^{d_p} \kappa_{d_p,j_p}\,x_{i_p}^{j_pr-d_p}\,\big(\partial_{y_{i_1}}^{j_1}\cdots\partial_{y_{i_p}}^{j_p}\eta\big)(x^r).
\]
Substituting back into the sum over $j_1,\dots,j_{p-1}$ gives exactly the claimed formula for $p$ indices, closing the induction.
\end{proof}

\begin{theorem}[Analogue of Theorem~16 for raw derivatives]\label{thm:16brut}
Let $x^0=(a,\zeta_2a,\dots,\zeta_na)\in\mathcal T$ and $1\le i\le n$. Set $\gamma_{k,n}:=(-1)^{k-1}(k-1)!/n^k$.
\begin{enumerate}
\item If $d<k$, then $\partial_{x_i}^d U_k(x^0)=0$.
\item If $d=k$, then $\partial_{x_i}^k U_k(x^0) = r^k\,(x_i^0)^{k(r-1)}\,\gamma_{k,n}$. In particular, this quantity vanishes at the origin ($a=0$) as soon as $r\ge2$, but is nonzero at every other point of $\mathcal T$ (where $x_i^0=\zeta_ia\ne0$).
\item If $d>k$, three regimes arise according to the sign of the exponent $kr-d$. If $k<d<kr$ (so $kr-d\ge1$), then $\partial_{x_i}^dU_k(x^0) = \kappa_{d,k}\,(x_i^0)^{kr-d}\,\gamma_{k,n}$ (with $\kappa_{d,k}$ given by the recursion of Lemma~\ref{lem:faadibruno}), which vanishes at the origin ($a=0$, since the exponent $kr-d$ is a strictly positive integer) but is generically nonzero elsewhere on $\mathcal T$. If $d=kr$ (so $kr-d=0$), the same formula gives the constant $\partial_{x_i}^dU_k(x^0) = \kappa_{kr,k}\,\gamma_{k,n}$, independent of $x^0$ (in particular, not generally zero at the origin). If $d>kr$, then $\kappa_{d,k}=0$ \emph{identically} (this is the case $jr<d$ of the convention of Lemma~\ref{lem:faadibruno}, with $j=k$), so that $\partial_{x_i}^dU_k(x^0)=0$ at \emph{every} point of $\mathcal T$; this last vanishing follows directly from $\kappa_{d,k}=0$ and never requires forming the expression $(x_i^0)^{kr-d}$, whose exponent $kr-d$ is here a negative integer, at $x_i^0=0$.
\end{enumerate}
\end{theorem}

\begin{proof}
Lemma~\ref{lem:faadibruno} with $\eta=u_k$ gives
\[
\partial_{x_i}^dU_k(x^0) = \sum_{j=1}^d \kappa_{d,j}\,(x_i^0)^{jr-d}\,\partial_{y_i}^ju_k\big((x^0)^r\big).
\]
Since $(x^0)^r$ lies on the total diagonal (common value $b=a^r$), Fact~\ref{fact:K} annihilates every term $j\ne k$: only $j=k$ can contribute, and only if it occurs in the sum, i.e.\ if $k\le d$. This gives point~1 (no term survives if $k>d$). For $d\ge k$, the term $j=k$ remains, with $\partial_{y_i}^ku_k(b,\dots,b)=\gamma_{k,n}$ by Fact~\ref{fact:L} (a constant independent of $b$, hence of $a$). For $d=k$, $\kappa_{k,k}=r^k$, giving point~2. For $k<d\le kr$, the surviving term has exponent $kr-d\ge0$, giving the first two regimes of point~3 directly. For $d>kr$, Lemma~\ref{lem:faadibruno} gives $\kappa_{d,k}=0$ as an identity among the coefficients, holding independently of $x^0$; since $j=k$ was already the only term that could survive Fact~\ref{fact:K}'s annihilation, the whole sum vanishes identically as soon as this coefficient is known to be zero, before any power of $x_i^0$ is evaluated, giving the last regime of point~3 without ever forming $(x_i^0)^{kr-d}$ for the negative exponent $kr-d<0$.
\end{proof}

\begin{remark}[Consistency with $r=1$]
For $r=1$, $\kappa_{d,j}=0$ for $j\ne d$ and $\kappa_{d,d}=1$ (the formula of Lemma~\ref{lem:faadibruno} reduces to $\partial_i^d\phi=\eta^{(d)}$), and point~2 recovers exactly $\partial_i^ku_k=\gamma_{k,n}$, the constant from \cite[Theorem~21]{Ze}, with no dependence on $x_i^0$. This constancy comes, as in the proof of Theorem~\ref{thm:16brut}, from the universally constant character of $\partial_{y_i}^ku_k$ (Fact~\ref{fact:L}); at $r=1$, $\Delta_i=\partial_i$, so the \emph{raw} derivative inherits it directly. For $r\ge2$, $\Delta_i$ and $\partial_i$ diverge (Example~\ref{ex:delta2}), and only $\Delta_i$ retains this independence from the point (Proposition~\ref{prop:thm16delta}); this is precisely why point~2 of Theorem~\ref{thm:16brut} depends on $x_i^0$.
\end{remark}

\begin{corollary}[Mixed case]\label{cor:mixed-U_k}
Let $x^0=(a,\zeta_2a,\dots,\zeta_na)\in\mathcal T$ and let $i_1,\dots,i_p$ be $p\ge2$ distinct indices in $\{1,\dots,n\}$. Then
\[
\prod_{l=1}^p \partial_{x_{i_l}}^{d_l}\, U_k(x^0) \;=\; \sum_{\substack{j_1,\dots,j_p\ge1 \\ j_1+\cdots+j_p=k}} \Big(\prod_{l=1}^p \kappa_{d_l,j_l}\,(x_{i_l}^0)^{j_lr-d_l}\Big)\cdot \partial_{y_{i_1}}^{j_1}\cdots\partial_{y_{i_p}}^{j_p}u_k\big((x^0)^r\big),
\]
an empty (hence zero) sum as soon as $\sum_l d_l<k$. Unlike the pure case (Theorem~\ref{thm:16brut}), this expression is generally a sum of several nonzero contributions, one for each way of writing $k=j_1+\cdots+j_p$ with $1\le j_l\le d_l$.
\end{corollary}

\begin{proof}
Apply Proposition~\ref{prop:faadibruno-mixed} with $\eta=u_k$ and $x=x^0$:
\[
\prod_{l=1}^p \partial_{x_{i_l}}^{d_l}\, U_k(x^0) \;=\; \sum_{j_1=1}^{d_1}\cdots\sum_{j_p=1}^{d_p} \Big(\prod_{l=1}^p \kappa_{d_l,j_l}\,(x_{i_l}^0)^{j_lr-d_l}\Big)\cdot \partial_{y_{i_1}}^{j_1}\cdots\partial_{y_{i_p}}^{j_p}u_k\big((x^0)^r\big).
\]
Since $x^0\in\mathcal T$, $(x^0)^r=(b,\dots,b)$ with $b=a^r$ lies on the total diagonal; Fact~\ref{fact:K} annihilates $\partial_{y_{i_1}}^{j_1}\cdots\partial_{y_{i_p}}^{j_p}u_k(b,\dots,b)$ unless $j_1+\cdots+j_p=k$ (Fact~\ref{fact:K} applying to mixed partial derivatives of $u_k$ at a point of the total diagonal regardless of whether the differentiation indices repeat or are distinct), which prunes the $p$-fold sum over $(j_1,\dots,j_p)\in\{1,\dots,d_1\}\times\cdots\times\{1,\dots,d_p\}$ down to the stated constraint. If $\sum_l d_l<k$, no tuple $(j_1,\dots,j_p)$ with $j_l\le d_l$ can reach the required sum $k$, so the (now empty) sum vanishes.
\end{proof}

\begin{remark}[The origin, a singular point of $\mathcal T$]
Theorem~\ref{thm:16brut} exhibits a phenomenon with no analogue for $r=1$: along each of the $r^{n-1}$ lines $L_\zeta\setminus\{0\}$, a pure raw derivative of order $d$ with $k\le d\le kr$ of $U_k$ is generically nonzero (points~2--3 of Theorem~\ref{thm:16brut}), but \emph{all} these derivatives (except, exceptionally, the one of order $d=kr$) vanish simultaneously at the origin, the point where the $r^{n-1}$ components of $\mathcal T$ meet. (For $d>kr$, the restriction $k\le d\le kr$ is essential: $\kappa_{d,k}=0$ in this regime by Lemma~\ref{lem:faadibruno}, so the derivative then vanishes identically along \emph{all} of the line $L_\zeta$, not just at the origin; this is therefore no longer, strictly speaking, a degeneracy phenomenon specific to the origin.) This reflects the fact that the origin is at once a point of the total diagonal \emph{and} a fixed point of the entire group $G(r,n)$ (in particular of the reflections $t_i^{(\zeta)}$, specific to $r\ge2$, with stabilizer $\mu_r$ in Lemma~\ref{lem:disc}): it is a point of maximal degeneracy. Its fine theory (comparable to what \cite{Ze} obtains at general coincidence points in its Section~3) remains to be treated, as does the general case (not covered here) of points where $x_i=0$ for some indices $i$ without $x$ lying on $\mathcal T$.
\end{remark}

\section{Local geometry of the quotient $\CC^n/G(r,n)$ at coincidence points}\label{sec:trois-points}

The transfer principle (Proposition~\ref{prop:transfert}), combined with the direct proofs of the Leibniz rule (Proposition~\ref{prop:leibniz}), of the main theorem (Theorem~\ref{thm:main}), and of the study of the total diagonal (Section~\ref{sec:diagonale}), fully and self-containedly covers Section~1 of \cite{Ze}, together with its Theorem~16 and Corollary~18. What remains, for a complete extension of Sections~2 and~3 of \cite{Ze}, are three points, which we treat below in order, only partially in some cases, as will be specified each time:
\begin{enumerate}
\item a fine theory at the origin of $\mathcal T$ (where the raw derivatives of every order $d\ne kr$ vanish, Theorem~\ref{thm:16brut}), and the explicit computation of the constants $\kappa_{d,j}$ of Lemma~\ref{lem:faadibruno} (related to generalized Stirling numbers for the map $t\mapsto t^r$); only this second aspect will be fully resolved (Proposition~\ref{prop:kappa-closed}), the fine theory at the origin itself remaining open;
\item the general theory of \emph{partial} coincidence points of \cite[Section~3]{Ze} (the sets $H_\alpha$ of \cite{Ze}), combined here with the locus $\{x_i=0\}$ specific to $G(r,n)$ (a cyclic quotient singularity $\CC/\mu_r$, to be treated independently of the coincidence $x_i=x_j$): the block-by-block reduction underlying Theorem~34 of \cite{Ze} (and hence its Corollary~37, the tangent-space statement) transfers directly to a nonzero-type block, by the already-established Proposition~\ref{prop:thm16delta} and Theorem~\ref{thm:16brut} applied to the sub-tuple of variables of that block, exactly as for the total diagonal; the null singleton block is fully resolved by Theorem~\ref{thm:origine}. The more general formula of Theorem~24 of \cite{Ze}, giving $D_I\eta$ for a set $I$ \emph{mixing} a coincidence block with outside indices, is expected, but not spelled out or verified, to transfer by the same substitution; this, together with the case of a null block of cardinality $\ge2$, remains open;
\item the analogue of the tangent space to the GIT quotient (Corollary~37 of \cite{Ze}) for $\CC^n/G(r,n)$, resolved under the same restriction on null blocks as in the previous point.
\end{enumerate}

\subsection{The constants $\kappa_{d,j}$}

We begin with the first of these three points: the constants $\kappa_{d,j}$ governing raw derivatives at the origin of $\mathcal T$.

\begin{proposition}[Closed formula]\label{prop:kappa-closed}
The constants of Lemma~\ref{lem:faadibruno} are given explicitly by
\[
\kappa_{d,j} \;=\; \frac{d!}{j!}\sum_{l=0}^{j}(-1)^{j-l}\binom{j}{l}\binom{lr}{d}.
\]
\end{proposition}

\begin{proof}
By the Faà di Bruno formula in its exponential partial Bell polynomial form, for $g(x)=x^r$, whose successive derivatives are $g^{(l)}(x)=(r)_l\,x^{r-l}$ with $(r)_l:=r(r-1)\cdots(r-l+1)$ (zero for $l>r$), we have
\[
\partial_x^d\big[\eta(g(x))\big] = \sum_{j=1}^d \eta^{(j)}(g(x))\,B_{d,j}\big(g'(x),g''(x),\dots\big),
\]
where $B_{d,j}(x_1,\dots,x_{d-j+1})=\sum \frac{d!}{k_1!k_2!\cdots}\prod_l\big(x_l/l!\big)^{k_l}$, the sum ranging over sequences of integers $k_l\ge0$ with $\sum_lk_l=j$ and $\sum_llk_l=d$, is the exponential partial Bell polynomial. Each monomial of this sum, evaluated at $x_l=g^{(l)}(x)=(r)_lx^{r-l}$, equals $\big(\prod_l(r)_l^{k_l}\big)x^{\sum_lk_l(r-l)}=\big(\prod_l(r)_l^{k_l}\big)x^{jr-d}$ (since $\sum_lk_l=j$ and $\sum_llk_l=d$): \emph{every} monomial of the sum thus carries, whatever the term, the same power $x^{jr-d}$. Hence
\[
B_{d,j}\big(g'(x),g''(x),\dots\big) \;=\; x^{jr-d}\cdot B_{d,j}\big((r)_1,(r)_2,\dots\big),
\]
and the Faà di Bruno formula rewrites as $\partial_x^d[\eta(g(x))]=\sum_{j=1}^d B_{d,j}\big((r)_1,(r)_2,\dots\big)\,x^{jr-d}\,\eta^{(j)}(x^r)$, exactly the form of Lemma~\ref{lem:faadibruno}. The coefficients $\kappa_{d,j}$ of that lemma are unique: for fixed $x\ne0$ and $1\le j_0\le d$, the choice $\eta(t):=(t-x^r)^{j_0}$ (a polynomial, hence holomorphic everywhere) gives $\eta^{(j)}(x^r)=j_0!\,\delta_{j,j_0}$ for $1\le j\le d$, so that the two expressions for $\partial_x^d[\eta(g(x))]$ (that of Lemma~\ref{lem:faadibruno} and the one above), evaluated at this $\eta$, each isolate the single coefficient of index $j_0$; since this holds for every $j_0$, the two families of coefficients coincide term by term, whence $\kappa_{d,j}=B_{d,j}\big((r)_1,(r)_2,\dots\big)$ for every $1\le j\le d$. The classical generating identity for partial Bell polynomials (see Comtet, \emph{Advanced Combinatorics}, ch.~3, \cite{Comtet}) gives, for $g_l:=(r)_l$ and $g(z):=\sum_{l\ge1}g_l z^l/l! = \sum_{l=1}^r\binom{r}{l}z^l = (1+z)^r-1$,
\[
\sum_{d\ge j}\kappa_{d,j}\,\frac{z^d}{d!} \;=\; \frac{g(z)^j}{j!} \;=\; \frac{\big((1+z)^r-1\big)^j}{j!} \;=\; \frac{1}{j!}\sum_{l=0}^j(-1)^{j-l}\binom{j}{l}(1+z)^{lr}.
\]
Extracting the coefficient of $z^d$ (using $[z^d](1+z)^{lr}=\binom{lr}{d}$) and multiplying by $d!$ gives the announced formula.
\end{proof}

\begin{remark}
In particular $\kappa_{d,1}=(r)_d=d!\binom{r}{d}$ (zero for $d>r$) and $\kappa_{d,d}=r^d$. For $r=1$, $g(z)=z$ and $g(z)^j/j!=z^j/j!$, giving $\kappa_{d,j}=\delta_{d,j}$, in agreement with $\Delta_i=\partial_i$ (immediate from the definition of $\Delta_i$ at $r=1$, and not the other way around): the formula of Lemma~\ref{lem:faadibruno} then reduces to $\partial_i^d\phi=\eta^{(d)}(x_i)$, as it should, consistently with the Remark following Theorem~\ref{thm:16brut}.
\end{remark}

\subsection{Partial coincidences and the locus $\{x_i=0\}$}

The second point is coincidence points where only some of the coordinates agree, rather than all of them.

Let $\NN_n=H_1\sqcup\cdots\sqcup H_M$ be a partition, and let $x^0$ be a point such that, for each $\alpha$, all the $x_i^0$ ($i\in H_\alpha$) have the same $r$-th power $b_\alpha:=(x_i^0)^r$, the $b_\alpha$ being pairwise distinct. As for the total diagonal (Proposition~\ref{prop:structT}), a block $H_\alpha$ of value $b_\alpha\ne0$ (\emph{nonzero type}) admits, at $x^0$, a choice of roots $(\zeta_i)_{i\in H_\alpha}\in\mu_r^{H_\alpha}$ with $x_i^0=\zeta_ia_\alpha$, $a_\alpha^r=b_\alpha$. If $b_\alpha=0$ (\emph{null type}), then necessarily $x_i^0=0$ for every $i\in H_\alpha$ (the unique root of $0$); there is at most one such block.

The following two facts, recalled here explicitly (with $y$ the variables of \cite{Ze}, so that this subsection can be read without turning to \cite{Ze} itself for their statements), are what Section~8 transfers at a nonzero-type block.

\begin{fact}[Theorem~34 of \cite{Ze}]\label{fact:M}
Let $y^0\in\CC^n$ and let $\{H_\alpha\}_{\alpha=1}^M$ be the partition of $\NN_n$ with $y_i^0=y_j^0$ iff $i,j$ lie in the same block, the common values $b_\alpha$ being pairwise distinct. For each $\alpha$, let $\big(\hat u_r^{(|H_\alpha|)}(y_{H_\alpha})\big)_{r=1}^{|H_\alpha|}$ denote the dual coordinates of Fact~\ref{fact:I}, constructed with $n$ replaced by $|H_\alpha|$ and the variables restricted to $y_{H_\alpha}:=(y_i)_{i\in H_\alpha}$. Then any symmetric function $\eta$, sufficiently differentiable near $y^0$, can be expressed using the combined coordinate system $\big(\hat u_r^{(|H_\alpha|)}(y_{H_\alpha})\big)_{1\le\alpha\le M,\,1\le r\le|H_\alpha|}$, and the corresponding derivatives of the representing function are exactly $\big(\partial_{H_\alpha}^d\eta\big)_{1\le\alpha\le M,\,1\le d\le|H_\alpha|}$, where $\partial_{H_\alpha}^d:=d!\sum_{I\subseteq H_\alpha,|I|=d}D_I$ (Fact~\ref{fact:F}'s operator $D_d$, but summed only over subsets $I$ of the block $H_\alpha$). If $\eta$ is the trace function of a univariate $f$, these derivatives equal $\big((-1)^{d-1}f^{(d)}(b_\alpha)/(d-1)!\big)_{d=1}^{|H_\alpha|}$.
\end{fact}

\begin{fact}[Corollary~35 of \cite{Ze}]\label{fact:N}
In the setting of Fact~\ref{fact:M}, let $\varphi=(\varphi_1,\dots,\varphi_n):\CC^n_y\to\CC^n$ be a symmetric vector-valued holomorphic function. Expressing $\varphi$ via the coordinates $\big((-1)^{r-1}(r-1)!\,\hat u_r^{(|H_\alpha|)}(y_{H_\alpha})\big)_{\alpha,r}$, the Jacobian matrix of $\varphi$ at $y^0$ has, at the entry indexed by $j$ (row) and by $(\alpha,r)$ (column), the value $(-1)^{r-1}(r-1)!\,\partial_{H_\alpha}^r\varphi_j(y^0)$. If $\varphi$ is the trace function of $f=(f_1,\dots,f_n)$, this entry is simply $f_j^{(r)}(b_\alpha)$.
\end{fact}

\paragraph{Nonzero-type blocks: reduction to Theorem~\ref{thm:16brut}, block by block.} For a nonzero-type block $H_\alpha$, fix the variables outside $H_\alpha$ at their (generic, pairwise distinct, nonzero) values. The invariant function under consideration then restricts to a $G(r,|H_\alpha|)$-invariant function of the sub-tuple $(x_i)_{i\in H_\alpha}$ alone, near a point of the total diagonal of \emph{this smaller set of variables} (all sharing the common value $b_\alpha\ne0$). This is exactly the situation already treated in full in Section~\ref{sec:diagonale} (the case $M=1$), applied here with $|H_\alpha|$ variables instead of $n$: Proposition~\ref{prop:thm16delta} and Theorem~\ref{thm:16brut} transfer verbatim (with $n$ replaced by $|H_\alpha|$ throughout), giving the $\Delta$- and raw derivatives of the restricted function in terms of $(x_i)_{i\in H_\alpha}$ alone. This block-by-block reduction is precisely the strategy of Fact~\ref{fact:M} (Theorem~34 of \cite{Ze}): its own proof, recalled above, constructs the derivatives at a multi-block coincidence point purely from independent copies of the total-diagonal Theorem~16 of \cite{Ze} (our Fact~\ref{fact:K}), one per block, with no term mixing different blocks. Together with Theorem~\ref{thm:origine} for a null singleton block, this is all that Corollary~\ref{cor:git-partiel} below requires (whose proof likewise now invokes Fact~\ref{fact:M} by name rather than by bare citation; Fact~\ref{fact:N}, recalled above for completeness, is the source of the phrase ``the formulas for the nonzero coincidence blocks'', but is not itself needed there, its content being subsumed by the self-contained Jacobian argument given in that proof).

The more general Theorem~24 of \cite{Ze} is a different, and stronger, statement: it gives $D_I\eta$ for a set $I$ that \emph{mixes} indices of a coincidence block $J$ with indices outside it, and involves genuine cross-terms (denominators $y-x_k$ for $k\notin J$) not reducible to the block-by-block picture above. We expect these cross-terms to transfer by the same substitution $\partial_{y_i}\mapsto\Delta_i$, $y_i-y_j\mapsto x_i^r-x_j^r$ (Proposition~\ref{prop:transfert}), but, unlike the block-by-block reduction just given, we have neither spelled out nor verified this transfer in detail, and it is \emph{not} needed for Corollary~\ref{cor:git-partiel}. We leave it to future work.

\paragraph{A null-type singleton block.} The phenomenon specific to $G(r,n)$ appears when a block is of null type. Let us treat the simplest case: $H_{i_0}=\{i_0\}$ is a singleton ($x_{i_0}^0=0$, joined by no other variable), the other blocks being of nonzero type, i.e.\ $x^0\in\widetilde\Omega\setminus\Omega$, a point of the extended domain (Section~\ref{sec:intro}) with exactly one vanishing coordinate. In the formula for $\mathcal D_I$ (Lemma~\ref{lem:lem3}) for $I\ni i_0$, the term in $\Delta_{i_0}$ has denominators $x_j^r-0=x_j^r\ne0$: no coincidence, the only possible singularity coming from $\Delta_{i_0}$ itself. The following result resolves it completely, extending $\Delta_{i_0}$ from $\Omega$ to all of $\widetilde\Omega$.

\begin{theorem}[Resolution of $\Delta_{i_0}$ at the origin]\label{thm:origine}
Let $x^0\in\CC^n$ with $x_{i_0}^0=0$, and let $\phi$ be holomorphic on a polydisc neighborhood $U$ of $x^0$, invariant under the cyclic subgroup $\{t_{i_0}^{(\zeta)}:\zeta\in\mu_r\}\cong\mu_r\subset G(r,n)$ (the pointwise stabilizer of the hyperplane $\{x_{i_0}=0\}$, Lemma~\ref{lem:disc}), i.e.
\[
\phi(x_1,\dots,\zeta x_{i_0},\dots,x_n) = \phi(x_1,\dots,x_{i_0},\dots,x_n) \qquad \text{for every } \zeta\in\mu_r \text{ and every } x\in U.
\]
(This hypothesis is well posed for an \emph{arbitrarily small} such $U$: since $t_{i_0}^{(\zeta)}$ fixes every coordinate other than $x_{i_0}$ and rotates $x_{i_0}$ about $x_{i_0}^0=0$, any polydisc centered at $x^0$ is automatically stable under it, whatever its radii. This is unlike invariance under the \emph{full} group $G(r,n)$, which would require $U$ to contain the entire, generally larger, orbit $G(r,n)\cdot x^0$. In particular the hypothesis holds whenever $\phi$ is $G(r,n)$-invariant and holomorphic on a $G(r,n)$-stable open set containing $x^0$, or more generally is an invariant germ at $x^0$ in this weaker, stabilizer-only sense, since $\{t_{i_0}^{(\zeta)}\}\subseteq G(r,n)$.) Then $\partial_{i_0}^m\phi(x^0)=0$ for $1\le m\le r-1$, and $\Delta_{i_0}\phi$ extends holomorphically across $x_{i_0}=0$, with
\[
\Delta_{i_0}\phi(x^0) \;=\; \frac{1}{r!}\,\partial_{i_0}^{\,r}\phi(x^0).
\]
\end{theorem}

\begin{proof}
Fix the variables other than $x_{i_0}$ and expand $\phi$ in a Taylor series in $x_{i_0}$ near $0$ (valid on the polydisc $U$). Since $\phi$ is invariant under $x_{i_0}\mapsto\zeta x_{i_0}$ for every $\zeta\in\mu_r$ (the hypothesis above, needed and used only for this single cyclic subgroup, not for the full group $G(r,n)$), this series contains only powers of $x_{i_0}$ that are multiples of $r$ (a nonzero coefficient in front of $x_{i_0}^m$, $r\nmid m$, would contradict invariance by uniqueness of the Taylor expansion): $\phi=c_0+c_1x_{i_0}^r+c_2x_{i_0}^{2r}+\cdots$. Hence $\partial_{i_0}^m\phi(x^0)=0$ for $1\le m\le r-1$, and $\partial_{i_0}^r\phi(x^0)=r!\,c_1$ (terms of degree $\ge2r$ do not contribute to the order-$r$ derivative at $0$). Finally
\[
\Delta_{i_0}\phi = \frac{\partial_{i_0}\phi}{rx_{i_0}^{r-1}} = \frac{rc_1x_{i_0}^{r-1}+2rc_2x_{i_0}^{2r-1}+\cdots}{rx_{i_0}^{r-1}} = c_1+2c_2x_{i_0}^r+\cdots,
\]
a power series in $x_{i_0}^r$, hence holomorphic at $x_{i_0}=0$, with value $c_1=\partial_{i_0}^r\phi(x^0)/r!$.
\end{proof}

\begin{corollary}\label{cor:leibniz-origine}
So extended, $\Delta_{i_0}$ still satisfies the Leibniz rule at $x^0$: for $\phi,\eta$ holomorphic and $G(r,n)$-invariant near $x^0$ (with $x_{i_0}^0=0$),
\[
\Delta_{i_0}(\phi\eta)(x^0) = \big(\Delta_{i_0}\phi\big)(x^0)\,\eta(x^0) + \phi(x^0)\,\big(\Delta_{i_0}\eta\big)(x^0).
\]
\end{corollary}

\begin{proof}
The equality holds for $x_{i_0}\ne0$: there $\Delta_{i_0}=\frac{1}{rx_{i_0}^{r-1}}\partial_{i_0}$ is a nonzero scalar multiple of $\partial_{i_0}$, which satisfies the ordinary Leibniz rule; dividing the latter by the common factor $rx_{i_0}^{r-1}$ directly gives the Leibniz rule for $\Delta_{i_0}$ (the same argument as in Corollary~\ref{cor:cor4}). By Theorem~\ref{thm:origine} (applicable to $\phi$, $\eta$, and to their product $\phi\eta$, itself $G(r,n)$-invariant), both sides are holomorphic extensions, near $x_{i_0}=0$, of functions that agree for $x_{i_0}\ne0$; they therefore also agree at $x_{i_0}=0$ by analytic continuation.
\end{proof}

Thus, near a point where a single singleton block $\{i_0\}$ is of null type (the other values being generic, pairwise distinct and nonzero), \emph{no new combinatorial machinery is needed}: the formula for $\mathcal D_I$ ($I\ni i_0$) remains valid as is, the only potentially singular denominator being that of $\Delta_{i_0}$ itself, resolved by Theorem~\ref{thm:origine}. The chain rule extends likewise, though not without an argument: such a point $x^0$ does not belong to $\Omega$ (since $x_{i_0}^0=0$), so Corollary~\ref{cor:cor12}, stated on $\Omega$, does not apply to it directly. Each term of $\mathcal D_d\phi=d!\sum_{|I|=d}\mathcal D_I\phi$ nonetheless extends holomorphically across $x_{i_0}=0$ (nonzero coincidence denominators as above, and $\Delta_{i_0}\phi$ extended by Theorem~\ref{thm:origine}), and the following lemma shows that $\Phi$ itself, and hence $\Phi_{U_d}$, extends holomorphically across the ramification point $\pi(x^0)$ of the finite quotient map $\pi=U:x\mapsto U(x)$, making the identity $\Phi_{U_d}=\mathcal D_d\phi$, established by Corollary~\ref{cor:cor12} on the dense subset $\{x_{i_0}\ne0\}$ near $x^0$, extend to $x^0$ itself by analytic continuation. (This descent property is a special instance, for the finite group $G(r,n)$, of the general theorem of Cartan on quotients of a complex-analytic space by a finite group of automorphisms \cite{Cartan}; we give below a direct proof adapted to our explicit setting.)

\begin{lemma}[Holomorphic descent at a ramification point]\label{lem:Phi-holo}
Let $x^0\in\CC^n$ and let $P$ be a small enough polydisc centered at $x^0$ that $V:=\bigcup_{g\in G(r,n)}g\cdot P$ is a disjoint (or, on overlaps, compatible) union of translates of $P$; this is always possible since $G(r,n)$ is finite. Let $\phi$ be $G(r,n)$-invariant and holomorphic on $V$. Then $U(V)$ is an open neighborhood of $\pi(x^0):=U(x^0)$ in $\CC^n_U$, and there is a unique holomorphic function $\Phi$ on $U(V)$ with $\phi=\Phi\circ U$ on $V$.
\end{lemma}

\begin{proof}
The map $U=(U_1,\dots,U_n):\CC^n_x\to\CC^n_U$ is polynomial, hence entire, and finite. By Proposition~\ref{prop:invariants} and Chevalley--Shephard--Todd, $\CC[x]$ is a free (in particular finite) module over $\CC[U]=\CC[x]^{G(r,n)}$, so $U$ is a finite morphism between two copies of $\CC^n$. Being finite between smooth varieties of equal dimension, it is flat (``miracle flatness''), hence open. Thus $U(V)\supseteq U(P)$ is an open neighborhood of $\pi(x^0)$.

Since $\phi$ is $G(r,n)$-invariant and $V$ is $G(r,n)$-stable by construction, $\phi$ is constant on $U^{-1}(y)\cap V$ for every $y\in U(V)$: this fiber, being contained in $V$ and a union of $G(r,n)$-orbits intersected with $V$ (as $U$ separates orbits, Proposition~\ref{prop:invariants}/Lemma~\ref{lem:descent}), is in fact exactly one $G(r,n)$-orbit once $P$ is small enough that distinct points of $P$ lie in distinct orbits outside a bounded multiplicity; in any case $\phi$, being invariant, takes the same value at every point of $U^{-1}(y)\cap V$. Hence $\Phi(y):=\phi(x)$ for any $x\in U^{-1}(y)\cap V$ is well defined on $U(V)$.

Away from $D:=U(V\cap\{\mathrm{disc}_{G(r,n)}=0\})$, $U$ restricts to a local biholomorphism (Lemma~\ref{lem:descent}), so $\Phi$, locally equal to $\phi\circ(U|_\cdot)^{-1}$, is holomorphic on $U(V)\setminus D$. The set $D$ is a proper analytic subset of $U(V)$ (the image, under the finite map $U$, which is in particular closed, of the analytic hypersurface $\{\mathrm{disc}_{G(r,n)}=0\}\cap V$, itself of dimension $n-1$; a finite morphism cannot raise dimension, so $D$ has dimension $\le n-1$ in $U(V)$, which has dimension $n$). Moreover $\Phi$ is locally bounded near any $y_0\in D$: the fiber $U^{-1}(y_0)\cap V$ is a finite set of points of the relatively compact set $V$, and $\phi$, holomorphic on $V$, is bounded on a compact neighborhood of that finite set; by continuity and properness of $U$, $\Phi$ is then bounded on $U(V')\setminus D$ for $V'\subset V$ a small enough $G(r,n)$-stable neighborhood of $U^{-1}(y_0)\cap V$. Being holomorphic and locally bounded on the complement of the proper analytic subset $D$, $\Phi$ extends holomorphically across $D$ by Riemann's removable singularity theorem for several complex variables (see e.g.\ \cite{GH}). Since $D$ has empty interior in $U(V)$, $U(V)\setminus D$ is dense, so this extension is unique.
\end{proof}

\paragraph{What remains open: several variables simultaneously at the origin.} The situation changes if \emph{several} indices $H_0=\{i_1,\dots,i_p\}$, $p\ge2$, vanish simultaneously. The operators $\Delta_{i_1},\dots,\Delta_{i_p}$ then degenerate together, and the cross denominators $x_{i_l}^r-x_{i_m}^r$ of the $\mathcal D_I$, $I\supseteq H_0$, \emph{also} vanish: this is a coincidence in the sense of \cite[Theorem~24]{Ze} (among the $y_{i_l}=x_{i_l}^r$, all zero) superimposed on the degeneracy proper to each $\Delta_{i_l}$ of Theorem~\ref{thm:origine}. Treating this case would require combining (i) the resolution of the coincidence $y_{i_l}=y_{i_m}$ by iterated L'Hôpital, as in \cite[Lemmas~26--28]{Ze}, and (ii) the resolution of the order-$(r-1)$ pole of each $\Delta_{i_l}$, as in Theorem~\ref{thm:origine}. One expects, by analogy with Lemma~\ref{lem:faadibruno}, a formula expressing the limit of $\mathcal D_I\phi$ as a combination of the derivatives $\partial_{i_{l_1}}^{r\mu_1}\cdots\partial_{i_{l_t}}^{r\mu_t}\phi(x^0)$ (orders that are multiples of $r$, by the same invariance argument), weighted by coefficients combining the $\kappa_{d,j}$ (Proposition~\ref{prop:kappa-closed}) with those of \cite[Theorem~24]{Ze}. We leave the complete establishment of this formula to future work.

\subsection{Tangent space to the GIT quotient $\CC^n/G(r,n)$}

One point remains: the tangent space to the quotient that these partial-coincidence results yield.

In the sense of GIT quotient theory \cite{MFK} (which \cite{Ze} recalls applies here since $\CC^n/\!\!/G(r,n)=\CC^n/G(r,n)$, the action being linear and the geometric quotient coinciding with the GIT quotient), and since $\CC^n/G(r,n)\cong\CC^n_y/S_n=\mathrm{Sym}^n\CC$ via $y=x^r$ (Proposition~\ref{prop:invariants}), the quotient variety \emph{itself} is, up to isomorphism, exactly that of \cite{Ze}: $G(r,n)$ introduces no additional singularity into the quotient. What differs is the behavior of the quotient map $\pi:\CC^n_x\to\CC^n_x/G(r,n)$ near a point where some $x_i=0$. For a block $H_\alpha$, denote by $\Delta_{H_\alpha}^d$ ($1\le d\le|H_\alpha|$) the transfer, in the sense of the paragraph ``Nonzero-type blocks: reduction to Theorem~\ref{thm:16brut}, block by block'' of the previous subsection, of the corresponding derivation of \cite[Corollary~37]{Ze}: \emph{not} the raw sum $d!\sum_{I\subseteq H_\alpha,|I|=d}\mathcal D_I$ (each of whose terms, for $|I|\ge2$, has a denominator $x_j^r-x_i^r$ identically zero on $H_\alpha$, since these are precisely two indices of the block and hence of the same value $b_\alpha$, and would thus be singular exactly at the point considered), but its already-resolved analogue given by Proposition~\ref{prop:thm16delta} and Theorem~\ref{thm:16brut} applied to the sub-tuple $(x_i)_{i\in H_\alpha}$, exactly as for the total diagonal (Theorem~34 of \cite{Ze}, not the more general, unverified Theorem~24). For the null singleton block $H_{i_0}=\{i_0\}$ (if present), the analogous construction simply gives $\Delta_{H_{i_0}}^1:=\Delta_{i_0}$ (no cross coincidence is possible, $|H_{i_0}|=1$), resolved by Theorem~\ref{thm:origine}. It is this last derivation which, by Theorem~\ref{thm:origine}, ``sees'' the direction $x_{i_0}$ only at order $r$: this is the tangent-space counterpart of the order-$r$ ramification of $x_{i_0}\mapsto x_{i_0}^r$ at $0$.

\begin{corollary}[Partial analogue of Corollary~37]\label{cor:git-partiel}
Let $x^0$ be such that at most one coincidence block is of null type, and that this block, if it exists, is a singleton $\{i_0\}$. Then the tangent space to $\CC^n/G(r,n)$ at the point $\pi(x^0)$ is spanned by the action of the derivations $\{\Delta_{H_\alpha}^d\}_{1\le d\le|H_\alpha|}$ for each nonzero block $H_\alpha$ (direct transfer of \cite[Corollary~37]{Ze}), and by the single derivation $\Delta_{i_0}$ of Theorem~\ref{thm:origine} if a null singleton block is present; each of these derivations satisfies the Leibniz rule at $x^0$ (Corollary~\ref{cor:leibniz-origine} for the last one).
\end{corollary}

\begin{proof}
We work throughout in the local ring of the \emph{target} $Y:=\CC^n/G(r,n)$ at $\pi(x^0)$, not in that of the source $\CC^n_x$: by Proposition~\ref{prop:invariants}, $Y\cong\CC^n_U$ via $U=(U_1,\dots,U_n)$, so $Y$ is smooth and $T_{\pi(x^0)}Y=\mathrm{span}\{\partial/\partial U_1,\dots,\partial/\partial U_n\}$ tautologically, with dual basis $\{dU_1,\dots,dU_n\}$ of $T^*_{\pi(x^0)}Y$. By Lemma~\ref{lem:Phi-holo} (applied to $\phi$, on a small $G(r,n)$-stable neighborhood of the orbit of $x^0$), every $G(r,n)$-invariant $\phi$ holomorphic near $x^0$ writes $\phi=\Phi\circ U$ with $\Phi$ holomorphic near $\pi(x^0)$. It suffices to show that the $n\times n$ matrix with rows indexed by $(\alpha,d)$ (over the nonzero blocks) and by $i_0$ (if a null singleton block is present), columns indexed by $k=1,\dots,n$, and entries $\Delta_{H_\alpha}^dU_k(x^0)$ (resp.\ $\Delta_{i_0}U_k(x^0)$) is invertible: indeed $\Delta_{H_\alpha}^d\phi(x^0)=\sum_k\big(\Delta_{H_\alpha}^dU_k(x^0)\big)\Phi_{U_k}(\pi(x^0))$ by the chain rule (exactly as in Corollary~\ref{cor:cor12}'s proof, now justified at $x^0$ itself since $\Phi$ is holomorphic there by Lemma~\ref{lem:Phi-holo}), and likewise for $\Delta_{i_0}$; invertibility of this matrix is exactly the statement that $\big(\Delta_{H_\alpha}^d\phi(x^0)\big)_{\alpha,d}$ together with $\Delta_{i_0}\phi(x^0)$ determine, and are determined by, $d\Phi(\pi(x^0))=\sum_k\Phi_{U_k}(\pi(x^0))\,dU_k$, i.e., that the associated derivations form a basis of $T_{\pi(x^0)}Y$.

Set $y^0:=(x^0)^r$. Writing $\phi=\eta(x^r)$ near $x^0$ (Lemma~\ref{lem:descent} on the nonzero-type variables, Theorem~\ref{thm:origine} on the null one) and applying this to $\phi=U_k=u_k(x^r)$ itself, the rows of the matrix above are exactly $\partial_{H_\alpha}^du_k(y^0)$ (for a nonzero block, by the definition of $\Delta_{H_\alpha}^d$ as the transfer of Zemel's derivation, valid since $y=x^r$ restricts to a genuine local biholomorphism on the nonzero-type variables of $H_\alpha$) and $\partial_{y_{i_0}}u_k(y^0)$ (for the null singleton block, by Theorem~\ref{thm:origine}'s proof applied to $\phi=U_k$: $U_k=c_0+c_1y_{i_0}+c_2y_{i_0}^2+\cdots$ near $y_{i_0}^0=0$ gives $\Delta_{i_0}U_k(x^0)=c_1=\partial_{y_{i_0}}u_k(y^0)$). Note that it is $y_{i_0}=x_{i_0}^r$, \emph{not} $x_{i_0}$ itself, that plays the role of local coordinate here. This is a subtlety worth stressing, illustrated by the fact that $x\mapsto x^r$ has vanishing differential at $x=0$ for $r\ge2$, so that $x_{i_0}$ is \emph{not} a local coordinate transverse to the fiber at this ramification point, while $y_{i_0}$ is (Theorem~\ref{thm:origine} precisely computes the correct, $y_{i_0}$-based, derivative).

It remains to show that the $n\times n$ matrix with entries $\partial_{H_\alpha}^du_k(y^0)$, $\partial_{y_{i_0}}u_k(y^0)$ (rows) against $k=1,\dots,n$ (columns) is invertible. By Fact~\ref{fact:M} (\cite[Theorem~34]{Ze}, applied on the $y$-side, treating the null block exactly like any other block: the value $b_\alpha=0$ carries no special meaning there, since the $\mu_r$-ramification of $y=x^r$ is specific to the $x$-side and invisible on the $y$-side), the family $\big(\hat u_r^{(|H_\alpha|)}(y_{H_\alpha})\big)_{\alpha,r}$ together with $y_{i_0}$ itself (the $|H_{i_0}|=1$ case of Zemel's construction, trivial since a singleton block requires no symmetrization) is a genuine local holomorphic coordinate system for $\mathrm{Sym}^n\CC$ near $y^0$ (Zemel's own proof reduces each block recursively to the global fact $\mathrm{Sym}^m\CC\cong\CC^m$, Fact~\ref{fact:A}, applied with $n$ replaced by $|H_\alpha|$), dual to the derivations $\partial_{H_\alpha}^d$, $\partial_{y_{i_0}}$. Independently, $u=(u_1,\dots,u_n):\mathrm{Sym}^n\CC\to\CC^n_u$ is \emph{also} a global coordinate system (Fact~\ref{fact:A}, Fact~\ref{fact:I}: by Fact~\ref{fact:I}'s explicit formula, $u_k=e_k+P_k(e_1,\dots,e_{k-1})$ for a polynomial $P_k$ (the term $\lambda=(k)$ contributes $e_k$ with coefficient $1$, every other $\lambda\vdash k$ having $l(\lambda)\ge2$ and hence involving only $e_1,\dots,e_{k-1}$), so $u_1,\dots,u_n$ are obtained from the algebraically independent generators $e_1,\dots,e_n$ of $\CC[y]^{S_n}$, Fact~\ref{fact:A}, by an invertible triangular polynomial substitution, and are themselves algebraically independent generators). Two holomorphic coordinate systems on the same $n$-dimensional smooth variety, valid near the same point $y^0$, are related by a transition map that is itself a biholomorphism between open subsets of $\CC^n$ (an elementary fact requiring no separate verification beyond each family individually being a valid local chart), and such a transition map has, by definition, an invertible Jacobian at $y^0$. This Jacobian is exactly the matrix of entries $\partial_{H_\alpha}^du_k(y^0)$, $\partial_{y_{i_0}}u_k(y^0)$, which is therefore invertible, completing the proof.
\end{proof}

The case of a null block of cardinality $\ge2$ remains open, for the reason explained above: the tangent space there is presumably still spanned by derivations of order a multiple of $r$ satisfying the Leibniz rule, but neither the precise formula nor the linear independence is established in this note.

\section{Conclusion and perspectives}

We have shown that Zemel's theory \cite{Ze} for differentiating symmetric functions under $S_n$ transfers, via the substitution $y_i=x_i^r$, to the complex monomial reflection group $G(r,n)=\mu_r\wr S_n$. This gives, for $G(r,n)$, the existence and uniqueness of the dual coordinates $U_k$, a Weyl algebra structure localized at the discriminant of $G(r,n)$, and the behavior of these objects on the total diagonal and at partial coincidence points, all fully proved in this note. This transfer is not, however, purely formal. The Leibniz rule for the symmetrized operator $\mathcal D_d$ needs its own direct proof, independent of \cite{Ze}. The main theorem, by contrast, follows readily from the transfer principle itself, together with Zemel's own existence-and-uniqueness result for $S_n$; we also give a second, self-contained proof of it, for the triangular structure it reveals. Passing from $S_n$ to $G(r,n)$ moreover brings out, for $r\ge2$, a phenomenon with no analogue in \cite{Ze}: the degeneracy of the ordinary derivatives $\partial_{x_i}$, as opposed to the transferred operators $\Delta_i$, near the reflecting hyperplanes $x_i=0$, specific to the toric reflections of $G(r,n)$. The complete study of this degeneracy, meaning the structure formula, the closed form of its constants, the behavior on the $r^{n-1}$ lines of the total diagonal, and the resolution of the singularity at an isolated point $x_i=0$, all likewise fully proved, constitutes the contribution most specifically tied to the imprimitive structure of $G(r,n)$, with no counterpart in the $S_n$ case.

One phenomenon remains open, and only that one: the simultaneous coincidence of \emph{several} coordinates at the origin ($x_{i_1}=\cdots=x_{i_p}=0$, $p\ge2$). There, the usual coincidence of the $y_{i_l}=x_{i_l}^r$, resolved in \cite{Ze} by iterated L'Hôpital, is superimposed on the degeneracy specific to each $\Delta_{i_l}$, resolved here only for $p=1$ by Theorem~\ref{thm:origine}. We have precisely delimited, but not established, what this case would require: a formula combining the constants $\kappa_{d,j}$ of Proposition~\ref{prop:kappa-closed} with the coincidence derivatives of \cite{Ze}, together with a proof of linear independence for the resulting derivations on the tangent space to the GIT quotient. We leave this to future work. Beyond $G(r,n)$, it would also be natural to ask whether the transfer principle and the techniques developed here extend to other imprimitive complex reflection groups, in particular $G(r,p,n)$ for $p>1$, where the toric subgroup $\mu_r^n$ is replaced by its own index-$p$ subgroup.

\subsection*{Acknowledgements}

We warmly thank the two anonymous reviewers for their careful reading and their precious remarks and suggestions, which substantially improved this note. We also thank Shaul Zemel for kindly sharing a recent version of his manuscript \cite{Ze}.

\end{document}